\documentclass{siamart190516}
\usepackage{amscd}
\usepackage{amssymb}
\usepackage{graphicx,epsfig,amsmath,amsfonts}
\usepackage{color}
\usepackage{multirow}
\usepackage{comment}
\usepackage{subfig}
\usepackage[shortlabels]{enumitem}

\allowdisplaybreaks

\numberwithin{equation}{section}
\newtheorem{preexample}{Example}[section]
\newenvironment{example}{\begin{preexample}\upshape}{\end{preexample}}

\newcommand{\beq}{\begin{equation}}
\newcommand{\eeq}{\end{equation}}
\newcommand{\be}{\begin{eqnarray}}
\newcommand{\ee}{\end{eqnarray}}
\newcommand{\bex}{\begin{eqnarray*}}
\newcommand{\eex}{\end{eqnarray*}}
\newcommand{\ba}{\begin{array}}
\newcommand{\ea}{\end{array}}
\def\dps{\displaystyle}

\title{Numerical analysis of a high-order scheme for nonlinear fractional differential equations with uniform accuracy\thanks{Submitted to the editors DATE.\funding{This research
was supported by National Natural Science Foundation of China (Grant numbers 11901135, 11961009),
Foundation of Guizhou Science and Technology Department (No. [2017]1086),
The first author would like to acknowledge the financial support by the China Scholarship Council (201708525037).
}}}
\author
{Junying Cao\thanks{School of Data Science and Information Engineering,
Guizhou Minzu University, 550025 Guiyang, China. (\email{caojunying@gzmu.edu.cn})
}
\and
Zhenning Cai\thanks{Department of Mathematics, National University of Singapore, Singapore 119076, Singapore.
(\email{matcz@nus.edu.sg})}
}

\begin{document}

\maketitle

\begin{abstract}
We introduce a high-order numerical scheme for fractional ordinary differential
equations with the Caputo derivative. The method is developed by dividing the
domain into a number of subintervals, and applying the quadratic interpolation
on each subinterval. The method is shown to be unconditionally stable, and for
general nonlinear equations, the uniform sharp numerical order $3-\nu$ can be
rigorously proven for sufficiently smooth solutions at all time steps. The
proof provides a general guide for proving the sharp order for higher-order
schemes in the nonlinear case. Some numerical examples are given to validate
our theoretical results.
\end{abstract}

\section {Introduction} \label{sec:0}
In the past decades, fractional differential equations have been studied
extensively by many researchers, due to its success in describing some physical
phenomena and chemical processes more accurately than integer order
differential equations \cite{miller1993,luk2013prop,kilb2006t,maindi2010}.
Like most classical differential equations, the exact solutions of fractional
order differential equations are usually not available to us. Even if
analytical solutions can be found, they usually appear in the form of series
and are difficult to evaluate. Therefore, the numerical study of fractional
differential equations has also inspired a number of excellent research works
such as \cite{guo2015,li2015n,garrappa2018,lub1986,dieelm2002,diethelm2006,
zeng2018s,jiang2017}.

In this work, we are interested in the following initial value problem: for
some $\nu\in (0,1)$, we would like to find $y(x)$ such that
\beq\label{pb}
 _0D^{\nu}_x y(x)=f(x,y(x)),\quad 0< x\leq T,
\eeq
subject to the initial condition $y(0)=y_0$. In \eqref{pb},
the operator $_0D^{\nu}_x$ is the Caputo derivative,
defined by
\beq\label{caputo}
 _0D^{\nu}_xy(x)=\int^x_0\omega_{1-\nu}(x-s)y'(s)ds,
\eeq
where $\omega_{1-\nu}$ is defined by
\beq\label{omega}
\omega_{1-\nu}(x)=x^{-\nu}/\Gamma(1-\nu)
\eeq
with $\Gamma(\cdot)$ being Euler's gamma function. The function
$\omega_{1-\nu}(x)$ acts as the convolutional kernel, which satisfies
\be\label{omg}
\int_{s}^{t}\omega_{\nu}(t-\mu)\omega_{1-\nu}(\mu-s)d\mu=\omega_1(t-s)=1,
  \qquad \forall 0<s<t<+\infty.
\ee
The numerical method for this equation has been extensively studied in the
context of linear partial differential equations. For example, the L1-type
schemes based on piecewise linear interpolation has been studied in
\cite{lin2007, Gao2017SJ}, where the numerical order is $2-\nu$. Higher-order
schemes can be achieved by using quadratic interpolation \cite{GSZ2014} or
Taylor expansion \cite{LWD2014}, and the convergence order can reach $3-\nu$
for smooth solutions. Generalization to $(r+1-\nu)$th-order schemes have been
stuided in \cite{cao2015high,li2016high} by Lagrange interpolation.  A common
problem in these methods is that the theorectical order of the solution at the
first time step can only achieve $2-\nu$, as is shown in the numerical analysis
in \cite{LWD2014}. Such a problem is also mentioned in \cite{LLZ2018}, where
the author uses a finer grid near the initial value to maintain the numerical
accuracy. Other related works include, but are not limited to, \cite{alikh2015,
lv2016error, vog2017, bafft2017, du2019}.

In principle, these methods can be directly generalized to nonlinear problems.
However, the analysis of convergence order on such methods for nonlinear
problems is less seen in the literature. In \cite{JCX2013}, the authors
converted the Caputo fractional derivative to the Volterra integral and proved
the order of accuracy $3+\nu$ for $0<\nu<1$ and $4$ for $\nu\geq 1$. A similar
technique is applied in \cite{nguyn2017}. In \cite{JLZ2018}, the authors
applied the L1 formula to the subdiffusion equation, and obtained the numerical
order $\nu$ due to the insufficient smoothness of the solution. The numerical
order $2-\nu$ is proven in \cite{Li2018, Liao2019}. However, theoretical proofs
of numerical schemes with order $3-\nu$ for nonlinear problems are rarely seen
in the literature. In \cite{Luo2017}, it is demonstrated that the
generalization of schemes with order $3-\nu$ for linear problems also works for
nonlinear problems, but the proof for nonlinear problems is given only for the
truncation error. Clearly, nonlinearity has caused significantly difficulty in
the numerical analysis, especially on the transition from the estimation of the
truncation error to the error of the solution.

The aim of this work is to introduce a new $(3-\nu)$th-order scheme for the
fractional differential equation \eqref{pb}. Our main contributions include:
\begin{itemize}
\item A new finite-difference approximation of the Caputo derivative is
  developed, which leads to a high-order numerical method for \eqref{pb}
  with uniform accuracy at all time steps.
\item The unconditional stability for the eigenvalue problem is proven
  rigorously.
\item A novel proof for the convergence order is proposed for the general
  nonlinear right-hand sides.
\end{itemize}
Our method is based on the block-by-block approach \cite{KA06,HTV11} commonly
used for integral equations \cite{YOA54,LIP85}. The retain the numerical order
at the first time step, the proposed scheme couples the solutions at first two
time steps. However, such a coupling is not required in the later steps. The
analysis of stability is complicated by these initial steps, which requires
close look at the structure of the solutions. The convergence analysis is based
on a novel technology that couples the idea of a recent work \cite{liao2019dis}
and the strategy we used in the proof of stability, so that the order $3-\nu$
can be achieved for sufficiently smooth solutions and general nonlinear
right-hand sides.

The rest of this paper is organized as follows. Our numerical scheme is
introduced in Section \ref{sec:2}. In Section \ref{sec:3}, we prove the
unconditional stability of our method. Section \ref{sec:4} is devoted to the
proof of the convergence order, as is verified by our numerical examples in
Section \ref{sec:5}. Finally, some concluding remarks are given in Section
\ref{sec:6}.

\section {A finite difference approximation to the Caputo derivative}
\label{sec:2}

In this section, we will construct an efficient numerical scheme for the problem \eqref{pb}.
For simplicity, we consider a uniform grid on $[0,T]$ defined by the grid
points $x_j=j\Delta x$, $j=0,1,2,\cdots,2N$, where $N$ is a positive integer,
and $\Delta x=\frac{T}{2N}$ is the grid size.  Below we are going to use the
short hand $y_i=y(x_i)$ and $f_i=f(x_i,y_i)$ for all $i=0,1,2,\cdots,2N$.

First, we propose a high-order approximation to the Caputo derivative $
_0D^{\nu}_xy(x)$ on grid points $x_i$ based on piecewise quadratic
interplation. To present the quadratic interpolation, we introduce the
following notation:
\be\label{I}
I_{[x_j,x_{j+2}]}y(x) =
  \varphi_{0,j}(x)y_j+\varphi_{1,j}(x)y_{j+1}+\varphi_{2,j}(x)y_{j+2},
\qquad j \in \mathbb{N},
\ee
where $\varphi_{i,j}(x),i=0,1,2$, are Lagrange interpolating polynoimals defined as
\bex
\varphi_{0,j}(x)=\frac{(x-x_{j+1})(x-x_{j+2})}{2\Delta x^2}, \quad
\varphi_{1,j}(x)=\frac{(x-x_{j})(x-x_{j+2})}{-\Delta x^2}, \quad
\varphi_{2,j}(x)=\frac{(x-x_j)(x-x_{j+1})}{2\Delta x^2}.
\eex
When $j = 1,2$, we approximate $_0D^{\nu}_xy(x_j)$ by $_0D^{\nu}_x (I_{[x_0,x_2]} y)(x_j)$:
\begin{align}
\label{cfut1b}
\begin{split}
 _0D^{\nu}_xy(x_1)&=\int^{x_1}_0 y'(s)\omega_{1-\nu}(x_1-s)
ds
\approx\int^{x_1}_0[I_{[x_0,x_2]}y(s)]'
\omega_{1-\nu}(x_1-s)
ds\\
&=A_1^{0,0}y_0+A_1^{1,0}y_{1}+A_1^{2,0}y_2
\end{split} \\
\label{cfut2b}
\begin{split}
 _0D^{\nu}_xy(x_2)&=\int^{x_2}_0y'(s)
 \omega_{1-\nu}(x_2-s)
ds
\approx\int^{x_2}_0[I_{[x_0,x_2]}y(s)]'
\omega_{1-\nu}(x_2-s)ds\\
&=A_2^{0,0}y_0+A_2^{1,0}y_{1}+A_2^{2,0}y_2
\end{split}
\end{align}
where
\bex
A_j^{i,0}=\int^{x_j}_0\varphi_{i,0}'(s)\omega_{1-\nu}(x_j-s)ds,
  \qquad i=0,1,2, \quad j = 1,2.
\eex

To approximate $_0D^{\nu}_xy(x_j)$ for $j > 2$, we assume that the values of
$y_0, y_1, \cdots, y_j$ are all given. Different approximations will be used
for odd and even $j$. When $j = 2m+1$, we approximate $y(x)$, $x \in
[0,x_{2m+1}]$ by
\bex
y(x) \approx \left\{ \begin{array}{ll}
  I_{[x_0,x_2]} y(x), & \text{if } x \in [x_0, x_1], \\
  I_{[x_{2k-1}, x_{2k+1}]} y(x), & \text{if } x \in [x_{2k-1}, x_{2k+1}], ~~ k = 1,2,\cdots,m.
\end{array} \right.
\eex
This suggests the following approach
\beq\label{cfut2m1b}
\begin{split}
& _0D^{\nu}_xy(x_{2m+1}) \\
= &\int^{x_{1}}_0y'(s)\omega_{1-\nu}(x_{2m+1}-s)ds
+\sum_{k=1}^m\int^{x_{2k+1}}_{x_{2k-1}}y'(s)\omega_{1-\nu}(x_{2m+1}-s)ds
\\
\approx & \int^{x_{1}}_0[I_{[x_0,x_2]}y(s)]'\omega_{1-\nu}(x_{2m+1}-s)ds
+\sum_{k=1}^m\int^{x_{2k+1}}_{x_{2k-1}}[I_{[x_{2k-1},x_{2k+1}]}y(s)]'
\omega_{1-\nu}(x_{2m+1}-s)ds
\\
= &\, A_{2m+1}^{0,0}y_0+A_{2m+1}^{1,0}y_{1}+A_{2m+1}^{2,0}y_2 +\sum_{k=1}^m
  \left( A_{2m+1}^{0,k}y_{2k-1}+A_{2m+1}^{1,k}y_{2k}+A_{2m+1}^{2,k}y_{2k+1} \right),
\end{split}
\eeq
where
\begin{align}
&A_{2m+1}^{i,0}=\int^{x_1}_0\varphi_{i,0}'(s)
\omega_{1-\nu}(x_{2m+1}-s)ds,\qquad i=0,1,2,\\
\label{A2mp1ik}
&A_{2m+1}^{i,k}=\int^{x_{2k+1}}_{x_{2k-1}}\varphi_{i,2k-1}'(s)
\omega_{1-\nu}(x_{2m+1}-s)ds, \quad i=0,1,2, \quad k=1,2,\cdots,m.
\end{align}
Similarly, when $j = 2m+2$, we approximate the Caputo derivative on $x_j$ based
on the following piecewise quadratic interpolation of $y(x)$:
\beq\label{ut2m2}
y(x)\approx I_{[x_{2k},x_{2k+2}]}y(x), \quad \forall x \in [x_{2k},x_{2k+2}],
  \quad k = 0,1,\cdots,m.
\eeq
As a consequence, $_0D^{\nu}_xy(x_{2m+2})$ can be approximated in the same way
as \eqref{cfut2m1b}, and the result is
\beq\label{cfut2m2b}
 _0D^{\nu}_xy(x_{2m+2}) \approx \sum_{k=0}^m
   \left( A_{2m+2}^{0,k}y_{2k}+A_{2m+2}^{1,k}y_{2k+1}+A_{2m+2}^{2,k}y_{2k+2} \right),
\eeq
where
\beq\label{a2m2io}
A_{2m+2}^{i,k}=\int^{x_{2k+2}}_{x_{2k}}\varphi_{i,2k}'(s)
\omega_{1-\nu}(x_{2m+2}-s)ds,\quad i=0,1,2, \quad k=0,1,\cdots,m.
\eeq

In all cases, the Caputo derivative $_0D^{\nu}_x y(x_j)$ is approximated by a
linear combination of $y_k$. Furthermore, by straightforward calculation, it
can be found that every $A_j^{i,k}$ is proportional to $\Delta x^{-\nu}$.
Therefore we summarize \eqref{cfut1b}\eqref{cfut2b}\eqref{cfut2m1b} and
\eqref{cfut2m2b} to write down them uniformly as
\beq \label{approx}
_0D^{\nu}_x y(x_j) \approx {}_0D_{\Delta x}^{\nu} y_j,
\eeq
where the newly introduced operator $_0D_{\Delta x}^{\nu}$ is the discrete
Caputo derivative defined by
\begin{displaymath}
_0D_{\Delta x}^{\nu} y_j = \left\{ \begin{array}{ll}
  {\Delta x}^{-\nu} \left( \widehat{D}_0y_0+\widehat{D}_1y_1+\widehat{D}_2y_2 \right), & \text{if } j = 1, \\[10pt]
  {\Delta x}^{-\nu} \left( \widetilde{D}_0y_0+\widetilde{D}_1y_1+\widetilde{D}_2y_2 \right), & \text{if } j = 2,\\[10pt]
  {\Delta x}^{-\nu} \dps\sum_{k=0}^{2m+1}D_{k}^{(m)}y_k, & \text{if } j=2m+1,~~m=1,2,\cdots,N-1,\\[15pt]
  {\Delta x}^{-\nu} \dps\sum_{k=0}^{2m+2}\overline{D}_{k}^{(m)}y_k, & \text{if } j=2m+2,~~m=1,2,\cdots,N-1.
\end{array} \right.
\end{displaymath}
Here all the coefficients ``$D$''s are constants depending only on $\nu$, and
their values can be computed analytically:
\begin{displaymath}
\begin{aligned}
& \widehat{D}_0 = \frac{3\nu-4}{2\Gamma(3-\nu)}, \qquad
  \widehat{D}_1 = \frac{2(1-\nu)}{\Gamma(3-\nu)}, \qquad
  \widehat{D}_2 = \frac{\nu}{2\Gamma(3-\nu)}, \\
& \widetilde{D}_0 = \frac{3\nu-2}{2^{\nu} \Gamma(3-\nu)}, \qquad
  \widetilde{D}_1 = -\frac{4\nu}{2^{\nu} \Gamma(3-\nu)}, \qquad
  \widetilde{D}_2 = \frac{\nu+2}{2^{\nu} \Gamma(3-\nu)}, \\
& \overline{D}_{0}^{(m)}=\frac{1}{\Gamma(3-\nu)}\left(-\frac{2-\nu}{2} \left[(2m)^{1-\nu}+3(2m+2)^{1-\nu}\right]-\left[(2m)^{2-\nu}-(2m+2)^{2-\nu}\right]\right),\\
& \overline{D}_{2k}^{(m)}=\frac{1}{\Gamma(3-\nu)} \bigg( -\frac{2-\nu}{2} \left[(2m-2k)^{1-\nu}+6(2m-2k+2)^{1-\nu}+(2m-2k+4)^{1-\nu}\right] \\
& \hspace{90pt} -\left[(2m-2k)^{2-\nu}-(2m-2k+4)^{2-\nu}\right]\bigg), \qquad k=1,2,\cdots,m,\\
& \overline{D}_{2k+1}^{(m)}=\frac{2}{\Gamma(3-\nu)} \bigg( (2-\nu)\left[(2m-2k)^{1-\nu}+(2m-2k+2)^{1-\nu}\right] \\
& \hspace{90pt} +\left[(2m-2k)^{2-\nu}-(2m-2k+2)^{2-\nu}\right]\bigg), \qquad k=0,1,\cdots,m, \\
& D_{0}^{(m)} = \frac{1}{\Gamma(3-\nu)}
  \left( \frac{2-\nu}{2}\left[(2m)^{1-\nu}-3(2m+1)^{1-\nu}\right]-(2m)^{2-\nu}+(2m+1)^{2-\nu} \right),\\
& D_{1}^{(m)} = \frac{1}{\Gamma(3-\nu)}
  \bigg( -\frac{2-\nu}{2}\left[(2m-2)^{1-\nu}+3(2m)^{1-\nu}-4(2m+1)^{1-\nu}\right] \\
& \hspace{90pt} -(2m-2)^{2-\nu}+3(2m)^{2-\nu}-2(2m+1)^{2-\nu} \bigg), \\
& D_{2}^{(m)} =\frac{1}{\Gamma(3-\nu)} \bigg( \frac{2-\nu}{2}\left[4(2m-2)^{1-\nu}+3(2m)^{1-\nu}-(2m+1)^{1-\nu}\right] \\
& \hspace{90pt} +2(2m-2)^{2-\nu}-3(2m)^{2-\nu}+(2m+1)^{2-\nu} \bigg),\\
& D_{2k}^{(m)}= \overline{D}_{2k+1}^{(m)}, \quad D_{2k-1}^{(m)} = \overline{D}_{2k}^{(m)}, \qquad k=2,3,\cdots,m, \\
\end{aligned}
\end{displaymath}

Based on the approximation \eqref{approx}, the numerical scheme for \eqref{pb}
with initial condition $y(0) = y_0$ can be written as
\beq\label{scheme}
_0D_{\Delta x}^{\nu}y_j = f(x_j, y_j), \qquad j=1,2,\cdots,2N.
\eeq
The above scheme is implicit. Since $_0D_{\Delta x}^{\nu}y_1$ depends on $y_2$, the values of $y_1$ and $y_2$ have to be solved simultaneously, which is the key to getting uniform accuracy without loss of precision at the first time step. For $j > 2$, solving $y_j$ needs only to solve a single equation.

\section{Stability analysis} \label{sec:3}

This section is devoted to the stability analysis of our numerical scheme.
Consider the fractional ordinary differential equation \eqref{pb} with
right-hand side
\beq\label{eigenpb}
f(x,y)=-\lambda y, \qquad \lambda > 0.
\eeq
In this case, the scheme \eqref{scheme} for $k > 2$ can be rewritten as
\beq \label{d1ud}
(1 + \tilde{\alpha}) y_j = \sum_{k=0}^{j-1} d_k^j y_k, \quad k = 3,4,\cdots,2N,
\eeq
where
\begin{subequations}\label{cxis}
\begin{align} \label{cxis1}
&d_{k}^{2m+1}=-\frac{D_{k}^{(m)}}{\alpha_0}, \qquad k=0,1,\cdots,2m,\\
&d_{k}^{2m+2}=-\frac{\overline{D}_{k}^{(m)}}{\alpha_0}, \qquad k=0,1,\cdots,2m+1,\\
&\alpha_0=D_{2m+1}^{(m)}=\overline{D}_{2m+2}^{(m)}=\frac{\nu+2}{\Gamma(3-\nu)2^{\nu}},
\qquad \tilde{\alpha}=\frac{\lambda{\Delta x}^{\nu}}{\alpha_0}>0.
\end{align}
\end{subequations}
Our purpose is to show that there exists a constant $K$ such that $|y_j| < K
|y_0|$ for any $j$. Such a property would be obvious from \eqref{d1ud} if all
the coefficients $d_k^j$ were positive. Unfortunately, this is not true for
some $\nu \in (0,1)$. The following lemma shows the properties of the
coefficients $d_k^j$:

\begin{lemma}\label{xis1}
For any $0<\nu<1$, $j\geq4$, the coefficients in the scheme \eqref{d1ud} satisfy
\begin{enumerate}[{\rm (1)}]
\item $\dps\sum_{k=0}^{j-1}d_{k}^{j}=1$.
\item $d_{k}^{j}> \dfrac{2\nu}{3\alpha_0 \Gamma(1-\nu)} (j-k)^{-\nu-1}$, $k=2,\cdots,j-3$.
\item $d_{j-1}^{j}>0$, $d_0^{j} > 0$, $d_1^{j} > 0$.
\item There exists $\nu_0 \in (0,1)$ such that $d_{j-2}^{j} > 0$ if $\nu \in (0,\nu_0)$,
  and $d_{j-2}^j < 0$ if $\nu \in (\nu_0, 1)$.
\item $\dfrac{1}{4}(d_{j-1}^{j})^2+d_{j-2}^{j}> \dfrac{2^{-\nu} \nu}{8\alpha_0 \Gamma(1-\nu)} > 0$.
\end{enumerate}
\end{lemma}
\begin{proof}
For simplicity, below we only present proof for the case $j = 2m+1$, $m \geq
2$. The proof for even $j$ is very similar. The statements below rely on some
technical inequalities, which are provided in Appendix \ref{sec:70}.

(1) By the fact that the scheme \eqref{scheme} for $j=2m+1$ is exact for constant solutions, we have
\bex
\sum_{k=0}^{2m}D_{k}^{(m)}+D_{2m+1}^{(m)}=0.
\eex
According to the definition in the \eqref{cxis}, we immediately obtain the equality of (1).

(2) For any $k = 2,\cdots,2m-2$, we let $g_k^{(m)} =  \Gamma(3-\nu)
D_k^{(m)}$. According to \eqref{cxis}, we have
\[
d_k^{2m+1}=-\frac{1}{\alpha_0\Gamma(3-\nu)}g_{k}^{(m)},
\]
and we are going to prove
\beq \label{gk}
g_k^{(m)} < -(2-\nu)(1-\nu)\nu (2m+1-k)^{-\nu-1}
\eeq
by considering the following three cases separately:

Case 1: $k=2$. In this case, we claim that
\begin{align*}
g_2^{(m)} &= \frac{2-\nu}{2} [4(2m-2)^{1-\nu}+3(2m)^{1-\nu}-(2m+1)^{1-\nu} ]\\
&\quad +[2(2m-2)^{2-\nu}-3(2m)^{2-\nu}+(2m+1)^{2-\nu}] \leq
  -\frac{5(2-\nu)(1-\nu)\nu}{4} (2m-1)^{-\nu-1}.
\end{align*}
To show this, we rewrite the above inequality by applying binomial expansion on both sides:
\begin{align*}
& \sum_{j=0}^{+\infty} (-1)^j {2-\nu \choose j+3}
  \frac{j+1}{2} \left( 2^{j+4} - (-1)^j \right) \left( \frac{1}{2m} \right)^{j+1+\nu} \\
\leq{} & \sum_{j=0}^{+\infty} (-1)^j {2-\nu \choose j+3}
  \frac{5}{4} (j+3)(j+2)(j+1) \left( \frac{1}{2m} \right)^{j+1+\nu}.
\end{align*}
This inequality holds if
\beq \label{et}
2^{j+4} - (-1)^j \geq \frac{5}{2} (j+3)(j+2), \qquad \forall j = 0,1,2,\cdots.
\eeq
When $j = 0$, this can be directly verified. When $j \geq 1$, let $h(x) =
2^{x+4} - 1 - \frac{5}{2} (x+3)(x+2)$. It can be easily verified that $h(x)$ is
convex when $x \geq 1$. Using $h'(1) > 0$ and $h(1) > 0$, we conclude that $h(x)$
is positive for all $x \geq 1$. Therefore \eqref{et} holds.

Case 2: $k=4,6,\cdots,2m-2$. In this case,
\[
g_k^{(m)} = 2\tilde{g}_{k'}^{(m)},
\]
where $k' = m - k/2$ and
\[
\tilde{g}_{k'}^{(m)}=(2-\nu)[(2k')^{1-\nu}+(2k'+2)^{1-\nu}]+(2k')^{2-\nu}-(2k'+2)^{2-\nu} .
\]
If $k' > 1$, by Lemma \ref{gineq}, we can obtain
\[
\tilde{g}_{k'}^{(m)} < -(2-\nu)(2k')^{1-\nu} \left( \frac{2}{2k'} \right)^2 \frac{(1-\nu)\nu}{6} \left( 1 - \frac{\nu+1}{2} \frac{2}{2k'} \right) \leq (3-\nu){2-\nu \choose 3} (2k')^{-\nu-1}.
\]
Thus
\[
g_{k}^{(m)} < 2(3-\nu) {2-\nu \choose 3} (2m-k)^{-\nu-1} < -\frac{2}{3} (2-\nu)(1-\nu)\nu (2m+1-k)^{-\nu-1}.
\]
When $k'=1$, by Lemma \ref{xis0}(7), we get
\[
\tilde{g}_{k'}^{(m)} = 2^{1-\nu} [4-\nu - (2+\nu) 2^{1-\nu}] < \frac{1}{27} (2\nu-3)(2-\nu)(1-\nu)\nu
  < -3^{-\nu-2} (2-\nu)(1-\nu)\nu,
\]
where we have used $3^{-\nu} < (3-2\nu)/3$, which comes from the convexity of the function $3^{-\nu}$. The above inequality implies that \eqref{gk} also holds for $k = 2m-2$.

Case 3: $k = 3,5,7,\cdots,2m-3$. In this case, we have
\[
g_k^{(m)} = \frac{1}{2}\hat{g}_{\bar{k}}^{(m)},
\]
where $\bar{k} = m-(k-1)/2$, and
\[
\hat{g}_{\bar{k}}^{(m)}=-(2-\nu)(2\bar{k})^{1-\nu}
\left[\left(1-\frac{1}{\bar{k}}\right)^{1-\nu}+6+\left(1+\frac{1}{\bar{k}}\right)^{1-\nu}\right]
-2(2\bar{k})^{2-\nu}\left[\left(1-\frac{1}{\bar{k}}\right)^{2-\nu}-\left(1+\frac{1}{\bar{k}}\right)^{2-\nu}\right].
\]
Since $\bar{k}\geq2$, we can apply Lemma \ref{xis0}(1)(2) to get
\begin{align*}
\hat{g}_{\bar{k}}^{(m)} &\leq -(2-\nu)(2\bar{k})^{1-\nu}\left( 8-\frac{1-\nu}{2\bar{k}^2} \left[2^\nu-\left(\frac{2}{3}\right)^\nu\right]\right)
  -2(2\bar{k})^{2-\nu} \left[-2(2-\nu)\frac{1}{\bar{k}}+\frac{(2-\nu)(1-\nu)\nu}{3\bar{k}^3}\right]\\
&=-(2\bar{k})^{1-\nu}\frac{(2-\nu)(1-\nu)}{6\bar{k}^2}
  \left( 8\nu-3\left[2^\nu-\left(\frac{2}{3}\right)^\nu\right]\right).
\end{align*}
Let $f(\nu) = 8\nu-3[2^\nu-(2/3)^\nu]$. Then
\begin{displaymath}
f''(\nu) = 3 \left[ \left(\frac{2}{3}\right)^{\nu} \left( \log \frac{2}{3} \right)^2 - 2^{\nu} (\log 2)^2 \right] < 0.
\end{displaymath}
Therefore $f(\nu) \geq f(0)(1-\nu) + f(1)\nu = 4\nu$. Thus
\[
g_{k}^{(m)} = \frac{1}{2} \hat{g}_{\bar{k}}^{(m)} \leq -\frac{4(2-\nu)(1-\nu)\nu}{3} (2\bar{k})^{-1-\nu}
  = -\frac{4(2-\nu)(1-\nu)\nu}{3} (2m+1-k)^{-1-\nu}.
\]

(3) All the three inequalities can be directly shown as follows:
\begin{align*}
d_{2m}^{2m+1} &=\dps-\frac{D_{2m}^{(m)}}{\alpha_0}
  =\frac{2\nu2^{1-\nu}}{\Gamma(3-\nu)\alpha_0}=\frac{4\nu}{\nu+2}>0, \\
d_1^{2m+1} &= -\frac{1}{\alpha_0 \Gamma(3-\nu)} \bigg(
  -\frac{2-\nu}{2}[(2m-2)^{1-\nu}+3(2m)^{1-\nu}-4(2m+1)^{1-\nu}] \\
& \qquad -[(2m-2)^{2-\nu}-3(2m)^{2-\nu}+2(2m+1)^{2-\nu}] \bigg) > 0,
  \hspace{60pt} \text{[Due to Lemma \ref{xis0}(3)]} \\
d_0^{2m+1} &= -\frac{1}{\alpha_0 \Gamma(3-\nu)} \left(
  \frac{2-\nu}{2}[(2m)^{1-\nu}-3(2m+1)^{1-\nu}]-(2m)^{2-\nu}+(2m+1)^{2-\nu}
\right) > 0. \\
& \hspace{325pt} \text{[Due to \eqref{lem}]}
\end{align*}

(4) Since
\[
d_{2m-1}^{2m+1} = -\frac{D_{2m-1}^{(m)}}{\alpha_0}
=\frac{2^{1-\nu}}{\Gamma(3-\nu)\alpha_0}[3(2-\nu)-(6+\nu)2^{-\nu}],
\]
the sign of $d_{2m-1}^{2m+1}$ is determined by the sign of $h(\nu) :=
3(2-\nu)-(6+\nu)2^{-\nu}$, which satisfies
\[
h''(\nu) = -2^{-\nu} (\log 2)^2 \left(6 - \frac{2}{\log 2} + \nu \right) < 0, \quad
h'(0) = 6\log 2 - 4 > 0, \quad h'(1) = \frac{7}{2} (\log 2 - 1) < 0.
\]
Therefore $h(\nu)$ first increases and then decreases. By $h(0) = 0$ and $h(1)
= -1/2$, we know that $h(\nu)$ has only one zero $\nu_0$ in $(0,1)$, and
$h(\nu) > 0$ if $\nu \in (0,\nu_0)$ and $h(\nu) < 0$ if $\nu \in (\nu_0,1)$,
which agress with the conclusion of the lemma.

(5) By Lemma \ref{xis0}(8), we directly have
\begin{align*}
\frac{1}{4}(d_{2m}^{2m+1})^2+d_{2m-1}^{2m+1}
&= \frac{2}{(2+\nu)^2} \left[12-\nu^2-(12+8\nu+\nu^2)2^{-\nu} \right] \\
&> \frac{(2-\nu)(1-\nu)\nu}{8(2+\nu)} = \frac{2^{-\nu} \nu}{8\alpha_0 \Gamma(1-\nu)},
\end{align*}
which completes the proof.
\end{proof}

The above lemma shows that for $\nu < \nu_0$, all the coefficients $d_k^j$ are
positive. In this case, as mentioned previously, the stability of the scheme
can be immediately obtained from \eqref{d1ud}. However, this does not hold when
$\nu > \nu_0$. To deal with this case, below we are going to rewrite the scheme
\eqref{d1ud} as equations with all positive coefficients. To this end, we
introduce
\[
\bar{y}_j = y_j - \theta y_{j-1}, \qquad \text{for all } j \geqslant 1, \qquad \bar{y}_0 = y_0,
\]
where $\theta = 2\nu / (2 + \nu)$. In fact, we have $\theta = \frac{1}{2}
d_{j-1}^{j}$ for all $j \geq 4$. Thus the numerical solution $y_j$ can be
represented by $\bar{y}_j$ through
\beq \label{yk}
y_j = \bar{y}_j + \theta y_{j-1} = \bar{y}_j + \theta \bar{y}_{j-1} + \theta^2 y_{j-2}
  = \cdots = \sum_{k=0}^j \theta^{j-k} \bar{y}_k.
\eeq
For $j \geq 4$, we can rewrite the scheme \eqref{d1ud} by subtracting both sides by $\theta y_{j-1}$:
\beq \label{y2mp1}
\bar{y}_j + \tilde{\alpha} y_j =
  \theta y_{j-1} + \sum_{k=0}^{j-2} d_k^j y_k =\theta \bar{y}_{j-1}+
  \sum_{k=0}^{j-2} \left( \theta^{j-k} + \sum_{k'=k}^{j-2} d_{k'}^{j} \theta^{k'-k} \right) \bar{y}_k,
\eeq
where we have inserted \eqref{yk} to write the right-hand side as functions of $\bar{y}_j$. By defining
\beq\label{db1}
\bar{d}_j^j = -1, \quad \bar{d}_k^j = \theta^{j-k}+\sum_{k'=k}^{j-2}d_{k'}^j\theta^{k'-k},\quad k=0,1,\cdots,j-1, \quad j \geq 4,
\eeq
the equation \eqref{y2mp1} becomes
\beq\label{yba1}
\bar{y}_j + \tilde{\alpha} y_j =\theta \bar{y}_{j-1}+ \sum_{k=0}^{j-2} \bar{d}_k^{j} \bar{y}_k, \quad j \geq 4.
\eeq
Note that the same equation does not hold for $j = 3$. When $j = 3$, we can use the same method to rewrite $\bar{y}_3 + \tilde{\alpha} y_3$ as a linear combination of $\bar{y}_0$, $\bar{y}_1$ and $\bar{y}_2$. The result is
\beq\label{yy3}
\bar{y}_3 + \tilde{\alpha} y_3 = \bar{d}_2^3 \bar{y}_2 + \bar{d}_1^3 \bar{y}_1 + \bar{d}_0^3 \bar{y}_0,
\eeq
where
\begin{align}
\label{d23def}
\bar{d}_2^3 &= d^3_2 - \theta = \frac{\nu + 6}{\nu + 2} - \frac{4+\nu}{\nu + 2} \left(\frac{2}{3}\right)^{\nu-1}, \\
\label{d13def}
\bar{d}_1^3 &= \bar{d}^3_2 \theta + d_1^3 = \frac{1}{(2+\nu)^2} \left[
  -\nu^2 - 12 + 3(\nu^2 + 2\nu + 4) \left( \frac{2}{3} \right)^{\nu}
\right], \\
\label{d03def}
\bar{d}_0^3 &= \bar{d}_1^3 \theta + d_0^3 = \frac{1}{2} \frac{(\nu-2)^2}{(\nu+2)^3}
  \left[ 4-2\nu+3\nu \left(\frac{2}{3}\right)^{\nu} \right].
\end{align}
Additionally, we define $\bar{d}_3^3 = -1$, so that for any $j \geq 3$, we have
\beq \label{ddbar1}
d_k^j = \bar{d}_k^j - \theta \bar{d}_{k+1}^j,
  \qquad \text{for all } k = 0,1,\cdots,j-1.
\eeq
The following lemma shows that in the new ``scheme'' \eqref{yba1}, all the
coefficients are positive:

\begin{lemma}\label{bkm2}
For $0<\nu<1$, the coefficients defined in \eqref{db1} satisfy
\begin{enumerate}[{\rm (1)}]
\item $0<\bar{d}_2^3 < \theta<\dfrac{2}{3}$;
\item $\bar{d}_k^j>0,\quad k=0,1,\cdots,j-2, \quad j \geq 3$;
\item $\dps\theta+\sum_{k=0}^{j-2}\bar{d}_k^j<1, \quad j \geq 3$.
\end{enumerate}
\end{lemma}
\begin{proof}
(1) Using the fact that $(2/3)^{\nu-1}$ is a convex function, we have $(2/3)^{\nu-1} < (3-\nu)/2$. Therefore by \eqref{d23def},
\[
\bar{d}_2^3 > \frac{\nu+6}{\nu+2} - \frac{4+\nu}{\nu+2} \frac{3-\nu}{2} = \frac{\nu(\nu+3)}{2(\nu+2)} > 0.
\]
The inequality $\bar{d}_2^3 < \theta$ is a direct result of Lemma \ref{xis0}(5), since
\[
\bar{d}_2^3 - \theta = \frac{1}{2+\nu} \left[6-\nu - \left(2 + \frac{\nu}{2}\right) 2^{\nu} 3^{1-\nu} \right].
\]
The fact that $\theta < 2/3$ is obvious since $\theta = 2\nu/(\nu + 2)$.

(2) When $j = 3$, by Lemma \ref{xis0}(4) and \eqref{d13def}, we immediately see that $\bar{d}_1^3 > 0$. The fact that $\bar{d}_0^3 > 0$ can be observed from
\[
\bar{d}_0^3 = \frac{1}{2} \frac{(\nu-2)^2}{(\nu+2)^3}
  \left[ 4-2\nu+3\nu \left(\frac{2}{3}\right)^{\nu} \right]
> \frac{1}{2} \frac{(\nu-2)^2}{(\nu+2)^3} (4-2\nu+2\nu) > 0.
\]
When $j > 3$, by \eqref{db1} and Lemma \ref{xis1}\,(5), we get
\beq\label{db21}
\bar{d}_{j-2}^{j}=\theta^{2}+d_{j-2}^{j}=\frac{1}{4}(d_{j-1}^j)^2+d_{j-2}^j>0.
\eeq
For other cases, we notice that \eqref{db1} implies the following recurrence relation of $\bar{d}_k^j$:
\beq\label{recur}
\bar{d}_k^j=\theta\bar{d}_{k+1}^j+d_{k}^j, \qquad k=0,1,\cdots,j-3.
\eeq
Since $d_k^j > 0$ for all $k = 0,1,\cdots,j-3$, the equation \eqref{recur} shows that $\bar{d}_{k+1}^j > 0$ implies $\bar{d}_k^j > 0$. Thus, by mathematical induction with the base case \eqref{db21}, we see that $\bar{d}_k^j>0$ for all $k=0,1,\cdots,j-2$.

(3) When $j = 3$, direct calculation yields
\[
\theta + \bar{d}_1^3 + \bar{d}_0^3 = 1 + \frac{1}{2(2+\nu)^3}
  \left[ -2\nu^3+12\nu^2-56\nu-48+3\left(\frac{2}{3}\right)^\nu(3\nu^3+4\nu^2+20\nu+16) \right].
\]
By Lemma \ref{xis0}(6), we see that the above quantity is less than $1$. When $j \geq 4$, we let $Q_j$ be the left-hand side of the inequality. It can be observed from \eqref{db1} that
\[
(1-\theta)Q_j=\theta(1-\theta^j) + \sum_{k=0}^{j-2} (1 - \theta^{k+1}) d_k^j.
\]
According to Lemma \ref{xis1}\,(2), we have
\[
(1-\theta)Q_j <\theta(1-\theta^j) + \sum_{k=0}^{j-3} d_k^j + (1 - \theta^{j-1}) d_{j-2}^j
  = \theta + \sum_{k=0}^{j-2} d_k^j - \theta^{j-1} (\theta^2 + d_{j-2}^j),
\]
where we have used \eqref{db21} at the last step. Now we apply Lemma \ref{xis1}\,(1) to get
\bex
(1-\theta)Q_j<\theta + 1 - d_{j-1}^{j} = \theta + 1 - 2\theta = 1-\theta,
\eex
which indicates $Q_j < 1$.
\end{proof}

Base on this lemma, we can show the stability for the numerical solution $\bar{y}_j$:
\begin{lemma}\label{b3km2}
For $0<\nu<1$, we have
\beq\label{yyb1234}
\bar{y}_{j}^2+\tilde{\alpha}y_{j}^2\leq y_0^2, \qquad \text{for all } j > 0.
\eeq
\end{lemma}
\begin{proof}
We first prove \eqref{yyb1234} for $j=1$. When $f(x,y) = -\lambda y$, the
scheme \eqref{scheme} for the first two steps is
\beq \label{y1y2}
\begin{aligned}
\widehat{D}_0y_0+\widehat{D}_1y_1+\widehat{D}_2y_2&=-\beta_0y_1, \\
\widetilde{D}_0y_0+\widetilde{D}_1y_1+\widetilde{D}_2y_2&=-\beta_0y_2,
\end{aligned}
\eeq
where $\beta_0 = \lambda \Delta x^{\nu}$. By solving the linear system, we can get
\beq \label{bary1}
\bar{y}_{1}^2+\tilde{\alpha}y_{1}^2 =
  \frac{\tilde{a}_0+\tilde{a}_1 \Gamma(3-\nu) \beta_0+\tilde{a}_2
    \Gamma(3-\nu)^2 \beta_0^2+\tilde{a}_3 \Gamma(3-\nu)^3 \beta_0^3+\tilde{a}_4 \Gamma(3-\nu)^4 \beta_0^4}
{\tilde{b}_0+\tilde{b}_1\Gamma(3-\nu) \beta_0+\tilde{b}_2\Gamma(3-\nu)^2 \beta_0^2+\tilde{b}_3\Gamma(3-\nu)^3 \beta_0^3+\tilde{b}_4\Gamma(3-\nu)^4 \beta_0^4}y_0^2,
\eeq
where the coefficients satisfy
\begin{displaymath}
\begin{aligned}
\tilde{a}_0 &= 16(2-\nu)^4 > 0, \\
\tilde{a}_1 &= 8(\nu-2)^2 \left[ 2^{\nu} (5 \nu ^2-8 \nu +12)-4 (\nu ^2+2 \nu) \right] \\
  & \geq 8(\nu-2)^2 \left[(1 + \nu \log 2) (5 \nu ^2-8 \nu +12)-4 (\nu ^2+2 \nu) \right] > 0, \\
\tilde{a}_2 &= 16 \nu ^2 (\nu +2)^2-2^{\nu+3}\nu (5 \nu ^3+4 \nu ^2-28 \nu +32)
  + 4^{\nu} (25 \nu^4-76 \nu ^3+148 \nu ^2-256 \nu +192) \\
  & \geq 16 \nu ^2 (\nu +2)^2-8(1+\nu)\nu (5 \nu ^3+4 \nu ^2-28 \nu +32) \\
  & \qquad + (1 + \nu \log 4) (25 \nu^4-76 \nu ^3+148 \nu ^2-256 \nu +192) > 0, \\
\tilde{a}_3 &= 2^{\nu} \left[ 32 \nu ^2 (\nu+2)-2^{\nu+3} \nu (5 \nu ^2-10 \nu +8)+4^{\nu} (\nu +2) (4-3 \nu)^2 \right] \\
  & \geq 2^{\nu} \left[ 32 \nu ^2 (\nu+2)-8(1+\nu) \nu (5 \nu ^2-10 \nu +8)+(1 + \nu \log 4) (\nu +2) (4-3 \nu)^2 \right] > 0,\\
\tilde{a}_4 &= 4^{\nu+2} \nu^2 > 0,
\end{aligned}
\end{displaymath}
and
\begin{displaymath}
\begin{aligned}
\tilde{b}_0 - \tilde{a}_0 &= 128\nu (2-\nu)^2 > 0, \\
\tilde{b}_1 - \tilde{a}_1 &= 8(2-\nu) [2^{\nu}(\nu^3-30\nu^2+28\nu-8)-2(\nu+2)(\nu^2-8\nu-4)] \\
  & \geq 8(2-\nu) 2^{\nu} [(\nu^3-30\nu^2+28\nu-8)-2(\nu^2-8\nu-4)] > 0, \\
\tilde{b}_2 - \tilde{a}_2 &= 4^{\nu} (-9 \nu ^4+108 \nu ^3-196 \nu ^2+192 \nu -128) \\
  & \quad +2^{\nu +3} (3 \nu ^4-8 \nu ^3-44 \nu ^2+48 \nu +32)-4 (\nu -2) (3 \nu +2) (\nu +2)^2 \\
  & \geq (1 + \nu \log 4)(-9 \nu ^4+108 \nu ^3-196 \nu ^2+192 \nu -128) \\
  & \quad + 8(3 \nu ^4-8 \nu ^3-44 \nu ^2+48 \nu +32) -4 (\nu -2) (3 \nu +2) (\nu +2)^2 > 0, \\
\tilde{b}_3 - \tilde{a}_3 &= 2^{\nu} \left[2^{\nu +3} (3 \nu ^3-16 \nu ^2+8 \nu +8)-4^{\nu} (\nu +2) (4-3 \nu )^2-8 (\nu -2) (\nu +2) (3 \nu +2)\right] \\
  &> 2^{\nu} \left[8 (3 \nu ^3-16 \nu ^2+8 \nu +8)-4 (\nu +2) (4-3 \nu )^2-8 (\nu -2) (\nu +2) (3 \nu +2)\right] > 0, \\
\tilde{b}_4 - \tilde{a}_4 &= 4^{1+\nu} (4 + 4\nu - 3\nu^2) > 0.
\end{aligned}
\end{displaymath}
In the above derivation, we have omitted the details on the determination of signs for all polynomials of degree less than or equal to $5$, which is elementary but tedious. Since $\beta_0 = \lambda \Delta x^{\nu} > 0$, these inequalities show that the coefficient of $y_0^2$ on the right-hand side of \eqref{bary1} is less than 1. Therefore $\bar{y}_1^2 + \tilde{\alpha} y_1^2 \leq y_0^2$.

To show \eqref{yyb1234} for $j=2$, we also solve the linear system \eqref{y1y2} to get
\[
\bar{y}_{2}^2+\tilde{\alpha}y_{2}^2 =
  \frac{\check{a}_0+\check{a}_1 \Gamma(3-\nu) \beta_0+\check{a}_2
    \Gamma(3-\nu)^2 \beta_0^2+\check{a}_3 \Gamma(3-\nu)^3 \beta_0^3}
{\check{b}_0+\check{b}_1\Gamma(3-\nu) \beta_0+\check{b}_2\Gamma(3-\nu)^2 \beta_0^2+\check{b}_3\Gamma(3-\nu)^3 \beta_0^3+\check{b}_4\Gamma(3-\nu)^4 \beta_0^4}y_0^2,
\]
and it can be similarly shown that
\[
\check{b}_l > \check{a}_l > 0, \quad l = 0,1,2,3,
  \qquad \text{and} \qquad \check{b}_4 > 0.
\]
Therefore \eqref{yyb1234} also holds for $k = 2$.

Next, we prove \eqref{yyb1234} for $j=3$.
Multiplying by $2\bar{y}_{3}$ on both sides of \eqref{yy3}, and using the identity
\beq\label{hds}
2y_j\bar{y}_{j}=(y_j+y_j)\bar{y}_{j}=(y_j+\bar{y}_{j}+\theta y_{j-1})(y_j-\theta y_{j-1})
=y_j^2+\bar{y}_{j}^2-\theta^2y_{j-1}^2,
\eeq
we get
\beq\label{2y3}
2\bar{y}_{3}^2+\tilde{\alpha} y_3^2+\tilde{\alpha}\bar{y}_{3}^2-\tilde{\alpha}\theta^2y_{2}^2
\leq\bar{d}_{2}^{3}\bar{y}_2^2+\bar{d}_{1}^{3}\bar{y}_1^2+\bar{d}_{0}^{3}y_0^2
+(\bar{d}_{2}^{3}+\bar{d}_{1}^{3}+\bar{d}_{0}^{3})\bar{y}_{3}^2.
\eeq
Applying Lemma \ref{bkm2}\,(1)(3) and the result \eqref{yyb1234} for $j=1,2$ to
the above inequality, we obtain
\[
\bar{y}_{3}^2+\tilde{\alpha} y_3^2
\leq\theta(\bar{y}_2^2+\tilde{\alpha}y_2^2)
+\bar{d}_{1}^{3}(\bar{y}_1^2+\tilde{\alpha}y_1^2)+\bar{d}_{0}^{3}y_0^2
\leq (\theta+\bar{d}_{1}^{3}+\bar{d}_{0}^{3}) y_0^2 \leq y_0^2.
\]
Therefore, we can obtain \eqref{yyb1234} for $j=3$.

When $j \geq 4$, we apply mathematical induction and assume that the result
holds for all cases up to $j - 1$. To show \eqref{yyb1234}, we multiply both
sides of \eqref{yba1} by $2\bar{y}_j$ and apply the identity \eqref{hds},
resulting in the following inequality:
\begin{align*}
2\bar{y}_j^2+\tilde{\alpha}y_j^2+\tilde{\alpha}\bar{y}_j^2-\tilde{\alpha}\theta^2y_{j-1}^2
& \leq \theta \bar{y}_{j-1}^2+\sum_{k=0}^{j-2}\bar{d}_k^j\bar{y}_k^2
+ \left(\theta+\sum_{k=0}^{j-2} \bar{d}_k^j\right)\bar{y}_j^2
\leq \theta \bar{y}_{j-1}^2+\sum_{k=0}^{j-2}\bar{d}_k^j\bar{y}_k^2 + \bar{y}_j^2,
\end{align*}
where Lemma \ref{bkm2}\,(5) has been applied at the last step. Some rearrangement yields
\[
\bar{y}_j^2+\tilde{\alpha}y_j^2
\leq\theta(\bar{y}_{j-1}^2+\tilde{\alpha}y_{j-1}^2)
+\sum_{k=1}^{j-2}\bar{d}_k^j(\bar{y}_{k}^2+\tilde{\alpha}y_{k}^2)+\bar{d}_0^j y_0^2.
\]
Now one can apply the inductive hypothesis to get
\[
\bar{y}_j^2+\tilde{\alpha}y_j^2
\leq \left( \theta+\sum_{k=1}^{j-2}\bar{d}_k^j+\bar{d}_0^j \right) y_0^2 \leq y_0^2.
\]
By the principle of mathematical induction, the inequality \eqref{yyb1234}
holds for all $j > 0$. \end{proof}

By now, we are ready to show the stability of the original numerical solution $y_k$:
\begin{theorem}\label{stab1}
The scheme \eqref{scheme} for the equation \eqref{pb} with $f$ given in \eqref{eigenpb} is stable in the sense that
\beq\label{yddwx}
|y_k|\le \frac{2+\nu}{2-\nu} |y_0|,\quad \text{for all } k > 0.
\eeq
\end{theorem}
\begin{proof}
By Lemma \ref{b3km2}, we can get
\beq\label{ybmj12}
|\bar{y}_j|\le|y_{0}|, \quad \text{for all } j > 0.
\eeq
Inserting this inequality to \eqref{yk} yields
\[
|y_{k}| \le \sum_{j=0}^{k}\theta^{k-j} |\bar{y}_j|
  \leq \sum_{j=0}^{k}\theta^{k-j} |y_0| \leq \frac{1}{1-\theta} |y_0|
  = \frac{2+\nu}{2-\nu} |y_0|,
\]
which completes the proof of the stability.
\end{proof}

\section{Convergence analysis}\label{sec:4}

Our convergence analysis follows the general idea of the recent work
\cite{liao2019dis}, which is parallel to the proof of $L^2$-stability of the
fractional ODE \eqref{pb} with respect to the initial data. However, our
analysis has to deal with the special processing of the first two time steps
and the non-positivity of the coefficients in the numerical scheme. For the
sake of clarity, we decompose our analysis into the following three
subsections. Before that, we make the following assumptions:
\begin{description}
\item[(H1)] The exact solution $y \in C^3([0,T])$;
\item[(H2)] The right-hand side $f(x,y)$ is Lipschitz continuous with respect
  to $y$:
  \beq \label{Lipschitz}
  |f(x,y^*) - f(x,y^{**})| < L |y^* - y^{**}|, \qquad \text{for any $y^*$ and $y^{**}$}.
  \eeq
\end{description}
In the following analysis, we will restrict ourselves to the numerical solution
exactly on $[0,T]$. Precisely, we suppose $2N \Delta x = T$ for a positive
integer $N$. For convenience, we define the numerical error by
\[
e_j = y(x_j) - y_j, \qquad j = 0,1,\cdots,2N,
\]
and $e_0 = 0$. Furthermore, by the hypothese (H2), we can find $L_j$ for $j =
1,2,\cdots,2N$, such that
\beq\label{fxiyi}
f(x_j,y(x_j))-f(x_j,y_j)=L_j(y(x_j)-y_j) \qquad \text{and} \qquad |L_j| \leq L.
\eeq

\subsection{Reformulation of the numerical scheme} \label{sec:reformulation}
Our first step is to rewrite our numerical scheme to better match the form of
the Caputo derivative \eqref{caputo}. To this end, we introduce the notation
\[
\nabla \psi_k = \psi_k - \psi_{k-1}, \qquad k > 0
\]
for any quantity $\psi_k$, as corresponds to the first-order derivative
appearing in the definition of the Caputo derivative. Furthermore, for any $n
\geq 2$ and $k = 0,1,\cdots,n-1$, we define
$\bar{B}_k^n$ as
\beq \label{Bbar}
\bar{B}_0^n = \Delta x^{-\nu} \alpha_0, \qquad
\bar{B}_{n-k}^n = \bar{B}_{n-k-1}^n - \Delta x^{-\nu} \alpha_0 \bar{d}_k^n, \quad k = 1,\cdots,n-1.
\eeq
By \eqref{ddbar1}, we have
\begin{align*}
{}_0D_{\Delta x}^{\nu}y_n &= \Delta x^{-\nu} \alpha_0 y_n - \Delta x^{-\nu} \alpha_0 \sum_{k=0}^{n-1} d_k^n y_k
  = \Delta x^{-\nu} \alpha_0 y_n - \Delta x^{-\nu} \alpha_0 \sum_{k=0}^{n-1} (\bar{d}_k^n - \theta \bar{d}_{k+1}^n) y_k \\
&= \Delta x^{-\nu} \alpha_0 (y_n + \theta \bar{d}_n^n y_{n-1}) -
  \Delta x^{-\nu} \alpha_0 \sum_{k=1}^{n-1} \bar{d}_k^n (y_k - \theta y_{k-1}) - \Delta x^{-\nu} \alpha_0 \bar{d}_0^n y_0 \\
&= \Delta x^{-\nu} \alpha_0 \bar{y}_n -
  \Delta x^{-\nu} \alpha_0 \sum_{k=1}^{n-1} \bar{d}_k^n \bar{y}_k - \Delta x^{-\nu} \alpha_0 \bar{d}_0^n y_0.
\end{align*}
Now we can apply the definition of $\bar{B}_k^n$ given in \eqref{Bbar} to rewrite the discrete fractional derivative as
\beq \label{ybar}
\begin{split}
{}_0D_{\Delta x}^{\nu}y_n
&= \bar{B}_0^n \bar{y}_n + \sum_{k=1}^{n-1} (\bar{B}_{n-k}^n - \bar{B}_{n-k-1}^n) \bar{y}_k - \Delta x^{-\nu} \alpha_0 \bar{d}_0^n y_0 \\
&= \sum_{k=1}^n \bar{B}_{n-k}^n \nabla \bar{y}_k + (\bar{B}_{n-1}^n - \Delta x^{-\nu} \alpha_0 \bar{d}_0^n) y_0.
\end{split}
\eeq
Similarly, if we define $\bar{y}(x) = y(x) - \theta y(x - \Delta x)$ and $\bar{y}(0) = y_0$, we have
\beq \label{ybar1}
{}_0D_{\Delta x}^{\nu}y(x_n) = \sum_{k=1}^n \bar{B}_{n-k}^n \nabla \bar{y}(x_k)
  + (\bar{B}_{n-1}^n - \Delta x^{-\nu} \alpha_0 \bar{d}_0^n) y_0.
\eeq
Our analysis will be based on such a form of the discrete Caputo derivative. The following lemma provides the lower bounds and the monotonicity of the coefficients.

\begin{lemma}\label{pib}
The coefficients $\bar{B}_k^n$ satisfy $\bar{B}_0^n >\bar{B}_1^n >\cdots>\bar{B}_{n-1}^n>0$ and
\beq\label{dyxj1}
\bar{B}_{k}^{n}\geq\frac{1}{\pi_B\Delta x}\int_{x_{n-k-1}}^{x_{n-k}}\omega_{1-\nu}(x_n-s)ds
=\frac{1}{\pi_B\Delta x^{\nu}\Gamma(2-\nu)}[(k+1)^{1-\nu}-k^{1-\nu}],
\eeq
where $\pi_B=9$.
\end{lemma}

\begin{proof}
The monotonicity of the coefficients $\bar{B}_k^n$ is obvious by the definition \eqref{Bbar} and the positivity of $\bar{d}_k^n$, and below we focus only on the proof of \eqref{dyxj1}.

When $k = 0$, we have
\[
\bar{B}_0^n = \Delta x^{-\nu} \alpha_0
  = \frac{\nu+2}{(2-\nu)2^{\nu}} \frac{1}{\Delta x^{\nu} \Gamma(2-\nu)}
  \geq  \frac{1}{\Delta x^{\nu} \Gamma(2-\nu)}.
\]
When $k > 0$, since $(k+1)^{1-\nu} - k^{1-\nu} \leq (1-\nu) k^{-\nu}$, we just need to show
\[
\bar{B}_{k}^{n}\geq \frac{1}{9\Delta x^{\nu}\Gamma(2-\nu)} (1-\nu) k^{-\nu}
  = \frac{1}{9 \Delta x^{\nu} \Gamma(1-\nu)} k^{-\nu}.
\]
Below we separate our proof into four cases.

\textbf{Case 1:} $n = 3$. By direct calculation, one can obtain
\begin{align*}
\bar{B}_1^3 &= \bar{B}_0^3 - \Delta x^{-\nu} \alpha_0 \bar{d}_2^3
  = \frac{2^{-\nu}}{\Delta x^{\nu} \Gamma(3-\nu)} \left[ \frac{3(\nu+4)}{2}
    \left( \frac{2}{3} \right)^{\nu} - 4 \right] \\
& \geq \frac{1 - \nu \log 2}{\Delta x^{\nu} \Gamma(3-\nu)} \left[ \frac{3(\nu+4)}{2}
    \left( 1 + \nu \log \frac{2}{3} \right) - 4 \right]
= \frac{1}{\Delta x^{\nu} \Gamma(1-\nu)} \left[
  1 + \frac{\nu(\beta_0 + \beta_1 \nu + \beta_2 \nu^2)}{2(2-\nu)(1-\nu)}
\right],
\end{align*}
where $\beta_0 = 9 - 12\log 3 + 8\log 2$, $\beta_1 = 3 \log 3(4 \log 2 - 1) -
12(\log 2)^2 - 2$, $\beta_2 = 3\log 2 (\log 3 - \log 2)$. It is not difficult
to check that $\beta_0 + \beta_1 \nu + \beta_2 \nu^2 > 0$ when $\nu \in (0,1)$.
Therefore
\[
\bar{B}_1^3 \geq \frac{1}{\Delta x^{\nu} \Gamma(1-\nu)}.
\]
Similarly, the case $k=2$ can be shown by
\begin{align*}
\bar{B}_2^3 &= \bar{B}_1^3 - \Delta x^{-\nu} \alpha_0 \bar{d}_1^3
  = \frac{2^{-\nu}}{\Delta x^{\nu} \Gamma(2-\nu)} \cdot \frac{1}{\nu+2}
    \left[ 2-\nu + \frac{3\nu}{2} \left( \frac{2}{3} \right)^{\nu}\right] \\
& \geq \frac{2^{-\nu}}{\Delta x^{\nu} \Gamma(2-\nu)} \cdot \frac{1}{\nu+2}
    \left[ 2-\nu + \frac{3\nu}{2} \left( 1 + \nu \log \frac{2}{3} \right)\right] \\
& = \frac{2^{-\nu}}{\Delta x^{\nu} \Gamma(1-\nu)} \left[
    1 + \frac{\nu[3 + (2 - 3\log 3 + 3\log 2)\nu]}{2(2+\nu)(1-\nu)}
  \right] \geq \frac{2^{-\nu}}{\Delta x^{\nu} \Gamma(1-\nu)}.
\end{align*}

\textbf{Case 2:} $n > 3$ and $k = n-1$. By defintion,
\[
\bar{B}_{n-1}^n = \bar{B}_0^n - \Delta x^{-\nu} \alpha_0 \sum_{k=1}^{n-1} \bar{d}_k^n
  = \Delta x^{-\nu} \alpha_0 \left( 1 - \sum_{k=1}^{n-1} \bar{d}_k^n \right).
\]
By Lemma \ref{bkm2}\,(3) and \eqref{ddbar1}, we can bound $\bar{B}_{n-1}^n$ by
\[
\bar{B}_{n-1}^n > \Delta x^{-\nu} \alpha_0 \bar{d}_0^n > \Delta x^{-\nu} \alpha_0 d_0^n.
\]
Now we consider odd and even $n$ separately. If $n = 2m+1$ and $m > 1$, by \eqref{cxis1},
\[
\bar{B}_{2m}^{2m+1} > \frac{1}{\Delta x^{\nu}\Gamma(3-\nu)} \left(
  \frac{2-\nu}{2}[3(2m+1)^{1-\nu}-(2m)^{1-\nu}]-[(2m+1)^{2-\nu}-(2m)^{2-\nu}]\right).
\]
Using the inequality \eqref{lem}, we see that
\[
\bar{B}_{2m}^{2m+1}\geq \frac{1}{\Delta x^{\nu} \Gamma(3-\nu)} (2-\nu)(1-\nu) (2m)^{-\nu}
  \left( 1 - \frac{7\nu}{24m} \right) > \frac{1}{2\Delta x^{\nu} \Gamma(1-\nu)} (2m)^{-\nu}.
\]
Similarly, when $n = 2m+2$ and $m \geq 1$, we have
\begin{align*}
\bar{B}_{2m+1}^{2m+2}&\geq \frac{1}{\Delta x^{\nu} \Gamma(3-\nu)} (2-\nu)(1-\nu) (2m)^{-\nu}
  \left( 1 - \frac{5\nu}{6m} \right)\\
  &=\frac{1}{\Delta x^{\nu} \Gamma(3-\nu)} (2-\nu)(1-\nu) (2m+1)^{-\nu}\left(\frac{2m}{2m+1} \right)^{-\nu}
  \left( 1 - \frac{5\nu}{6m} \right)\\
   &> \frac{1}{6\Delta x^{\nu} \Gamma(1-\nu)} (2m+1)^{-\nu}.
\end{align*}

\textbf{Case 3:} $n > 3$ and $k = n - 2$. We can directly use the result for $k = n-1$ to get
\[
\bar{B}_{n-2}^n \geq \bar{B}_{n-1}^n >
  \frac{1}{6\Delta x^{\nu} \Gamma(1-\nu)} (n-2)^{-\nu} \cdot \left( \frac{n-1}{n-2} \right)^{-\nu}
  > \frac{1}{9\Delta x^{\nu} \Gamma(1-\nu)} (n-2)^{-\nu}.
\]

\textbf{Case 4:} $n > 3$ and $k = 1,2,\cdots,n-3$. Using
\begin{gather*}
\bar{d}_j^n = \theta \bar{d}_{j+1}^n + d_j^n \geq d_j^n
  > \frac{2\nu}{3\alpha_0 \Gamma(1-\nu)} (n-j)^{-\nu-1}, \quad j = 1,2,\cdots,n-3, \\
\bar{d}_{n-2}^{n} = \frac{1}{4} (d_{n-1}^{n})^2 + d_{n-2}^{n}
  > \frac{\nu}{4\alpha_0 \Gamma(1-\nu)} 2^{-\nu-1},
\end{gather*}
we obtain
\begin{align*}
\bar{B}_k^n &= \bar{B}_{n-1}^n + \Delta x^{-\nu} \alpha_0 \sum_{j=1}^{n-1-k} \bar{d}_j^n
  > \bar{B}_{n-1}^n +
    \frac{\nu}{4\Delta x^{\nu} \Gamma(1-\nu)} \sum_{j=2}^{n-1-k} (n-j)^{-\nu-1} \\
  &\geq \frac{1}{2\Delta x^{\nu} \Gamma(1-\nu)} (n-1)^{-\nu} +
    \frac{\nu}{4\Delta x^{\nu} \Gamma(1-\nu)} \int_1^{n-1-k} (n-x)^{-\nu-1} \,\mathrm{d}x \\
  &= \frac{1}{2\Delta x^{\nu} \Gamma(1-\nu)} (n-1)^{-\nu} +
    \frac{1}{4\Delta x^{\nu} \Gamma(1-\nu)} [(1+k)^{-\nu} - (n-1)^{-\nu}] \\
  & \geq \frac{1}{4\Delta x^{\nu} \Gamma(1-\nu)} (1+k)^{-\nu} \geq \frac{1}{8\Delta x^{\nu} \Gamma(1-\nu)} k^{-\nu}.
\end{align*}
This completes the proof for all $k = 0,1,\cdots,n-1$.
\end{proof}

The purpose of the above lemma is an upper bound for the discrete fractional derivative of $|\bar{e}_j|^2$. We state the result in the following lemma:
\begin{lemma} \label{pdc}
For any $j \geq 3$,
\[
2\bar{e}_j \sum_{k=3}^j \bar{B}_{j-k}^j \nabla \bar{e}_k
  \geq \sum_{k=3}^j \bar{B}_{j-k}^j \nabla (|\bar{e}_k|^2).
\]
\end{lemma}
We refer the readers to \cite[Lemma A.1]{liao2019dis} for the proof of this lemma.

\subsection{\bf Estimation of the truncation errors} \label{sec:4.2}

Most error estimation is based on the estimation of the truncation error. In our case, it can be defined by
\beq\label{r2m2t}
r_j(\Delta x) := {}_0D^{\nu}_xy(x_{j})-{}_0D^{\nu}_{\Delta x}y(x_j),\qquad j \geq 1.
\eeq
Here ${}_0D^{\nu}_{\Delta x}y(x_j)$ is defined by replacing $y_j$ in \eqref{cfut1b}--\eqref{cfut2m1b} and \eqref{cfut2m2b} with $y(x_j)$. As mentioned previously, the first two time steps in our scheme have to be taken into account independently. Therefore we introduce the following modified truncation error for $j \geq 3$:
\beq \label{trj}
\tilde{r}_j(\Delta x) = r_j(\Delta x) -
  L_j \sum_{k=1}^2 \theta^{j-k} \bar{e}_k + \sum_{k=1}^2 \bar{B}_{j-k}^j \nabla \bar{e}_k,
\eeq
where $\bar{e}_k = e_k - \theta e_{k-1}$ and. Below we are going to derive bounds for both \eqref{r2m2t} and \eqref{trj}.

\begin{theorem}\label{error}
Assume that (H1) holds. There exists a constant $C_1$ depending only on the function $y$ and the final time $T$, such that for all $\Delta x > 0$,
\beq\label{rch2m2t}
|r_j(\Delta x)|\leq C_1{\Delta x}^{3-\nu}, \qquad j = 1,2,\cdots,2N.
\eeq
\end{theorem}
\begin{proof}
Our error estimation will be established on the following error term of the Lagrange interpolation:
\be\label{ut01}
y(x)-I_{[x_k,x_{k+2}]}y(x)=\frac{y^{(3)}(\xi_k(x))}{6}(x-x_k)(x-x_{k+1})(x-x_{k+2}), \qquad \forall x\in[x_k,x_{k+2}],
\ee
where $\xi_k(x)$ is a function defined on $[x_k, x_{k+2}]$ with range $(x_k, x_{k+2})$. Let $M_1$ be the upper bound of $y^{(3)}$ on $[0,T]$. For any $x \in [x_k, x_{k+2}]$, we have
\beq\label{taylor2}
|y(x) - I_{[x_k,x_{k+2}]}y(x)| \leq \frac{M_1}{3} \Delta x^2 (x_{k+2} - x),
\eeq
or more simply,
\beq\label{taylor1}
|y(x) - I_{[x_k,x_{k+2}]}y(x)| \leq \frac{M_1}{6} \Delta x (x_{k+2}-x)(x-x_k) \leq \frac{M_1}{6}\Delta x^3.
\eeq
We first estimate $r_1(\Delta x)$:
\beq\label{I00}
\begin{split}
|r_1(\Delta x)|
&=\left|\frac{1}{\Gamma(1-\nu)}\int^{x_1}_0y'(s)
(x_1-s)^{-\nu}
ds
-\frac{1}{\Gamma(1-\nu)}\int^{x_1}_0[I_{[x_0,x_2]}y(s)]'
(x_1-s)^{-\nu}
ds \right|\\
&=\frac{\nu}{\Gamma(1-\nu)}\left|\int^{x_1}_0 [y(s) - I_{[x_0,x_2]} y(s)] (x_1-s)^{-\nu-1} ds\right| \\
&=\frac{\nu}{\Gamma(1-\nu)}\left|
\int^{x_1}_0
\frac{y^{(3)}(\xi_0(s))}{6}(s-x_0)(x_1-s)^{-\nu}(s-x_{2})ds\right|\\
&\leq\frac{\nu}{\Gamma(1-\nu)} \frac{M_1}{6}
\int^{x_1}_0 s(x_{2}-s)(x_1-s)^{-\nu}ds
=\frac{\nu}{(3-\nu) \Gamma(2-\nu)}\frac{M_1}{3}{\Delta x}^{3-\nu} < M_1 {\Delta x}^{3-\nu},
\end{split}
\eeq
where we have used $\Gamma(2-\nu) > 2/3$. The equation \eqref{I00} proves
\eqref{rch2m2t} for $j=1$. The case $j=2$ can be similarly proven, and here we
omit the details.

Now we estimate $r_{2m+1}(\Delta x)$ for $m \ge 1$. In a similar way to
\eqref{I00}, we can use integration by parts to obtain
\beq\label{ryut2m1}
\begin{split}
|r_{2m+1}(\Delta x)| &=
  \bigg|\frac{\nu}{\Gamma(1-\nu)}\int^{x_1}_0[y(s)-I_{[x_0,x_2]}y(s)]
(x_{2m+1}-s)^{-\nu-1}
ds\\
&\quad +\frac{\nu}{\Gamma(1-\nu)}\sum_{k=1}^{m-1}\int^{x_{2k+1}}_{x_{2k-1}}[y(s)-I_{[x_{2k-1},x_{2k+1}]}y(s)]
(x_{2m+1}-s)^{-\nu-1}
ds \\
&\quad +\frac{\nu}{\Gamma(1-\nu)} \int^{x_{2m+1}}_{x_{2m-1}}[y(s)-I_{[x_{2m-1},x_{2m+1}]}y(s)]
(x_{2m+1}-s)^{-\nu-1}
ds\bigg|.
\end{split}
\eeq
Applying \eqref{taylor2} and \eqref{taylor1}, we can estimate the truncation error by
\begin{align*}
|r_{2m+1}(\Delta x)| & \leq \frac{\nu}{\Gamma(1-\nu)} \left[
  \frac{M_1}{6} \Delta x^3 \left(
    \int_0^{x_1} (x_{2m+1} - s)^{-\nu-1} ds +
    \sum_{k=1}^{m-1} \int_{x_{2k-1}}^{x_{2k+1}} (x_{2m+1}-s)^{-\nu-1} ds
  \right)
\right] \\
&\quad +\frac{\nu}{\Gamma(1-\nu)} \int^{x_{2m+1}}_{x_{2m-1}}
  \frac{M_1}{3} \Delta x^2 (x_{2m+1}-s)^{-\nu} ds \\
& \leq \frac{M_1 \nu}{\Gamma(1-\nu)} \left[
  \frac{1}{6} \Delta x^3 \int_0^{x_{2m-1}} (x_{2m+1}-s)^{-\nu-1} ds +
  \frac{1}{3} \Delta x^2 \int_{x_{2m-1}}^{x_{2m+1}} (x_{2m+1}-s)^{-\nu} ds
\right] \\
& = \frac{M_1}{6\Gamma(2-\nu)}
  \left[ 2^{-\nu} (1 + 3\nu) - (1-\nu) (2m+1)^{-\nu} \right]
  \Delta x^{3-\nu} \leq M_1 \Delta x^{3-\nu}.
\end{align*}
The case $j=2m+2$ can be similarly proven, and the details are omitted.
\end{proof}

To show the error bounds for \eqref{trj}, we need the error estimation for the
first two time steps:

\begin{lemma} \label{e1e2}
Assume that both (H1) and (H2) hold, and $\Delta x^{\nu} < (7L)^{-1}$. Then
\[
|e_1| \leqslant 10C_1 \Delta x^3, \qquad |e_2| \leqslant 10C_1 \Delta x^3,
\]
where the constant $C_1$ is defined in Theorem \ref{error}.
\end{lemma}

\begin{proof}
By the numerical scheme \eqref{scheme} for $k = 1,2$ and the definition of the truncation error \eqref{r2m2t}, we have
\[
\left\{
\ba{l}
{\Delta x}^{-\nu}\widehat{D}_1e_1+{\Delta x}^{-\nu}\widehat{D}_2e_2=f(x_1,y(x_1))-f(x_1,y_1)-r_1(\Delta x), \\
{\Delta x}^{-\nu}\widetilde{D}_1e_1+{\Delta x}^{-\nu}\widetilde{D}_2e_2=f(x_2,y(x_2))-f(x_2,y_2)-r_2(\Delta x).
\ea
\right.
\]
After solving the equation, we get
\begin{align*}
|e_1| &= \Gamma(2-\nu) \Delta x^{\nu} \bigg|
  \frac{2+\nu}{2} [f(x_1, y(x_1)) - f(x_1,y_1) - r_1(\Delta x)] \\
& \quad - 2^{\nu-2} \nu [f(x_2, y(x_2)) - f(x_2,y_2) - r_2(\Delta x)] \bigg| \\
& \leq \Gamma(2-\nu) \Delta x^{\nu} \left(
  \frac{2+\nu}{2} (L|e_1| + C_1 \Delta x^{3-\nu}) +
  2^{\nu-2} \nu (L|e_2| + C_1 \Delta x^{3-\nu})
\right) \\
& \leq \Delta x^{\nu} \left(
  \frac{3}{2} (L|e_1| + C_1 \Delta x^{3-\nu}) +
  \frac{1}{2} (L|e_2| + C_1 \Delta x^{3-\nu})
\right) \leq \frac{3}{2} L \Delta x^{\nu} (|e_1| + |e_2|) + 2C_1 \Delta x^3.
\end{align*}
By similar means, we can obtain $|e_2| \leq 2L \Delta x^{\nu} (|e_1| + |e_2|) +
3C_1 \Delta x^3$. Summing up the two inequalities yields
\begin{displaymath}
|e_1| + |e_2| \leq \frac{7}{2} L \Delta x^{\nu} (|e_1| + |e_2|) + 5 C_1 \Delta x^3.
\end{displaymath}
Therefore when $\Delta x^{\nu} < (7L)^{-1}$, we have
\beq\label{e1pe2}
|e_1| + |e_2| \leq 10 C_1 \Delta x^3,
\eeq
which completes the proof.
\end{proof}

The above lemma already shows that we do not lose any numerical accuracy for
the first two time steps. In fact, their orders are slightly higher than the
general error bound $O(\Delta x^{3-\nu})$. This is necessary to provide error
bounds for $\tilde{r}_j(\Delta x)$ in the following theorem.
\begin{theorem}
Assume both (H1) and (H2) hold, and $\Delta x^{\nu} < (10|L|)^{-1}$. There exists a constant $C$ such that
\beq \label{rjdx}
|\tilde{r}_j(\Delta x)| \leq C \Delta x^{3-\nu}.
\eeq
\end{theorem}
\begin{proof}
We first estimate the coefficients $\bar{B}_{j-1}^j$ and $\bar{B}_{j-2}^j$. According to
\eqref{Bbar},
\begin{align*}
& \bar{B}_{j-1}^j < \Delta x^{-\nu} \alpha_0 =
  \frac{(\nu+2) \Delta x^{-\nu}}{\Gamma(3-\nu)2^{\nu}} \leq 3 \Delta x^{-\nu}, \\
& \bar{B}_{j-2}^j - \bar{B}_{j-1}^j = \Delta x^{-\nu} \alpha_0 \bar{d}_1^j
  < \Delta x^{-\nu} \alpha_0 \leq 3 \Delta x^{-\nu},
\end{align*}
where we have used $\bar{d}_1^j < 1$ implied by Lemma \ref{bkm2}(3)(5). Now
we can apply triangle inequality to \eqref{trj}:
\begin{align*}
|\tilde{r}_j(\Delta x)| & \leq |r_j(\Delta x)| +
  |(1+\theta) \bar{B}_{j-2}^j - \bar{B}_{j-1}^j| \cdot |e_1| +
  |L_j \theta^{j-2} - \bar{B}_{j-2}^j| \cdot |e_2| \\
& \leq C_1 \Delta x^{3-\nu} + 70 C_1 \Delta x^{3-\nu} +
  10 C_1 (|L_j| + 6 \Delta x^{-\nu}) \Delta x^3.
\end{align*}
Here we have applied Theorem \ref{error} and Lemma \ref{e1e2}. Since $\Delta
x^{\nu} < (10 |L|)^{-1}$, the above inequality yields
\[
|\tilde{r}_j(\Delta x)| \leq 132 C_1 \Delta x^{3-\nu}.
\]
\end{proof}

\subsection{Error analysis} \label{sec:7}
Now we are ready to summarize the previous two subsections and carry out the
error analysis for our scheme. The purpose of Section \ref{sec:reformulation}
is to provide prepartory works to introduce an important tool --- the
complementary discrete convolution kernels. Inspired by the property
\eqref{omg}, we would like to find the discrete kernel $P_j^n$, corresponding
to the kernel $\omega_{\nu}(\cdot)$, which satisfies
\beq\label{pnj}
\sum_{j=m}^{n}P_{n-j}^{n}\bar{B}_{j-m}^{j}\equiv1, \qquad \forall 3\leq m\leq n\leq 2N.
\eeq
According \cite[eq. (2.6)]{liao2019dis}, we have
\be\label{Pjn}
P_{0}^{n}=\frac{1}{\bar{B}_{0}^{n}}, \qquad
P_{j}^{n}=\frac{1}{\bar{B}_{0}^{n-j}}\sum_{k=0}^{j-1}(\bar{B}_{j-k-1}^{n-k}-\bar{B}_{j-k}^{n-k})P_{k}^{n}
\quad \text{for } 1\leq j\leq n-3.
\ee
Define
\beq \label{pnj2j1}
P_{n-2}^{n}=0, \qquad P_{n-1}^{n}=0.
\eeq
Then when $m=1,2$, we have
\begin{align} \label{mdengyu1}
\sum_{j=1}^{n}P_{n-j}^{n}\bar{B}_{j-1}^{j}
& =P_{n-1}^{n}\bar{B}_{0}^{j}+P_{n-2}^{n}\bar{B}_{1}^{j}
+\sum_{j=3}^{n}P_{n-j}^{n}\bar{B}_{j-1}^{j}\leq\sum_{j=3}^{n}P_{n-j}^{n}\bar{B}_{j-3}^{j}=1,\\\label{m2dengyu1}
\sum_{j=2}^{n}P_{n-j}^{n}\bar{B}_{j-2}^{j}
& =P_{n-2}^{n}\bar{B}_{0}^{j}+\sum_{j=3}^{n}P_{n-j}^{n}\bar{B}_{j-2}^{j}
\leq\sum_{j=3}^{n}P_{n-j}^{n}\bar{B}_{j-3}^{j}=1.
\end{align}

By Lemma \ref{pib} and \eqref{Pjn}--\eqref{pnj2j1}, we know that all the
coefficients $P_j^n\geq0$.  These coefficients help us ``invert'' the discrete
fractional derivative, so that we can derive the recursive inequality for the
numerical error:
\begin{lemma}
For any $n \geq 3$, it holds that
\beq \label{bare_n}
|\bar{e}_n|^2 \leq \sum_{j=3}^n P_{n-j}^n
  \sum_{k=3}^j 4L \theta^{j-k} |\bar{e}_k|^2 + |\bar{e}_2|^2 + 2\sum_{j=3}^n P_{n-j}^n |\bar{e}_j| \cdot |\tilde{r}_j(\Delta x)|,
\eeq
where $\tilde{r}_j(\Delta x)$ is defined in \eqref{trj}.
\end{lemma}

\begin{proof}
Plugging \eqref{ybar1} and \eqref{pb} into \eqref{r2m2t}, we get
\begin{align*}
r_j(\Delta x) &= f(x_j, y(x_j)) - \sum_{k=1}^j \bar{B}_{j-k}^j \nabla \bar{y}(x_k)
  - (\bar{B}_{n-1}^n - \Delta x^{-\nu} \alpha_0 \bar{d}_0^n) y_0 \\
&= f(x_j, y(x_j)) - \sum_{k=1}^j \bar{B}_{j-k}^j \nabla \bar{e}_k
  - \sum_{k=1}^j \bar{B}_{j-k}^j \nabla \bar{y}_k
  - (\bar{B}_{n-1}^n - \Delta x^{-\nu} \alpha_0 \bar{d}_0^n) y_0.
\end{align*}
By \eqref{ybar}\eqref{scheme} and \eqref{fxiyi}, the above equation can be further simplified:
\begin{align*}
r_j(\Delta x) &= f(x_j, y(x_j)) - f(x_j, y_j) - \sum_{k=1}^j \bar{B}_{j-k}^j \nabla \bar{e}_k
  = L_j e_j - \sum_{k=1}^j \bar{B}_{j-k}^j \nabla \bar{e}_k \\
& = L_j \sum_{k=1}^j \theta^{j-k} \bar{e}_k - \sum_{k=1}^j \bar{B}_{j-k}^j \nabla \bar{e}_k.
\end{align*}
Now we use the \eqref{trj} to rewrite the above equation as
\beq \label{barB}
\sum_{k=3}^j \bar{B}_{j-k}^j \nabla \bar{e}_k = L_j \sum_{k=3}^j \theta^{j-k} \bar{e}_k
  - \tilde{r}_j(\Delta x).
\eeq
Now we multiply both sides of the above equation by $2\bar{e}_j$. The right-hand side can be bounded by
\begin{align*}
2\bar{e}_j \left[ L_j \sum_{k=3}^j \theta^{j-k} \bar{e}_k
  - \tilde{r}_j(\Delta x) \right] &\leq
L \sum_{k=3}^j \theta^{j-k} (|\bar{e}_j|^2 + |\bar{e}_k|^2)
  + 2|\bar{e}_j| \cdot |\tilde{r}_j(\Delta x)| \\
& \leq \sum_{k=3}^j 4L \theta^{j-k} |\bar{e}_k|^2
  + 2|\bar{e}_j| \cdot |\tilde{r}_j(\Delta x)|,
\end{align*}
where we have used $\theta < 2/3$, and the left-hand side can be bounded from below by Lemma \ref{pdc}. Catenating both bounds using \eqref{barB}, we see that
\[
\sum_{k=3}^j 4L \theta^{j-k} |\bar{e}_k|^2 +
  2|\bar{e}_j| \cdot |\tilde{r}_j(\Delta x)| \geq
  \sum_{k=3}^j \bar{B}_{j-k}^j \nabla (|\bar{e}_k|^2).
\]
Multiplying both sides of the above equation by $P_{n-j}^j$ and taking the sum over $j$, one gets
\[
\sum_{j=3}^n P_{n-j}^j \sum_{k=3}^j 4L \theta^{j-k} |\bar{e}_k|^2 +
  2\sum_{j=3}^n P_{n-j}^j |\bar{e}_j| \cdot |\tilde{r}_j(\Delta x)| \geq
\sum_{j=3}^n P_{n-j}^j \sum_{k=3}^j \bar{B}_{j-k}^j \nabla (|\bar{e}_k|^2).
\]
Applying the identity \eqref{pnj} yields
\[
\sum_{j=3}^n P_{n-j}^j \sum_{k=3}^j 4L \theta^{j-k} |\bar{e}_k|^2 +
  2\sum_{j=3}^n P_{n-j}^j |\bar{e}_j| \cdot |\tilde{r}_j(\Delta x)| \geq
  \sum_{k=3}^n \nabla(|\bar{e}_k|^2) = |\bar{e}_n|^2 - |\bar{e}_2|^2,
\]
which is clearly equivalent to the conclusion of the lemma \eqref{bare_n}.
\end{proof}

The next step, we can now apply mathematical induction to bound the error by the initial error
and the truncation errors.

\begin{lemma} \label{ebar}
Let $\bar{e}_n = e_n - \theta e_{n-1}$ with $\theta = 2\nu / (2+\nu)$. If
\beq \label{deltax}
\Delta x^{\nu} \leq \frac{1}{24\pi_B L},
\eeq
it holds that
\beq\label{ebarn}
|\bar{e}_n| \leq 2E_{\nu}(24\pi_B L x_n^{\nu}) \left(
  |\bar{e}_2| + 2 \max_{3 \leq k \leq n} \sum_{j=3}^k P_{k-j}^k |\tilde{r}_j(\Delta x)|
\right), \qquad \text{for all } n \geq 2,
\eeq
where $\tilde{r}_j(\Delta x)$ is defined in \eqref{trj}, and $E_{\nu}$ is the Mittag-Leffler function defined by \eqref{ml}.
\end{lemma}

\begin{proof}
In the following proof, we need some useful properties of the kernel $P_j^n$
provided in Appendix \ref{sec:80}, wherein the complete details can be found.
Here we simply make references to the equations to be used.

For simplicity, we define
\[
F_n = 2E_{\nu}(24\pi_B x_n^{\nu}), \qquad G_n =
  |\bar{e}_2| + 2 \max_{3 \leq k \leq n} \sum_{j=3}^k P_{k-j}^k |\tilde{r}_j(\Delta x)|.
\]
Then both $F_n$ and $G_n$ are monotonically increasing with respect to $n$. Below we are going to prove the lemma using mathematical induction. Since $E_{\nu}(z) > 1$ for all $z > 0$, it is obvious that \eqref{ebarn} holds for $n = 2$. Now we assume that $n > 2$ and the estimation \eqref{ebarn} holds for all
$\bar{e}_2, \bar{e}_3, \cdots, \bar{e}_{n-1}$. Let
\[
|\bar{e}_{k(n)}| = \max_{2 \leq j \leq n-1} |\bar{e}_j|.
\]
If $|\bar{e}_n| \leq |\bar{e}_{k(n)}|$, then the monotonicity of $F_n$ and $G_n$ shows that
\[
|\bar{e}_n| \leq |\bar{e}_{k(n)}| \leq F_{k(n)} G_{k(n)} \leq F_n G_n.
\]
If $|\bar{e}_n| > |\bar{e}_{k(n)}|$, then by the inequality \eqref{bare_n},
\beq \label{en}
|\bar{e}_n|^2 \leq |\bar{e}_n| \left( \sum_{j=3}^{n-1} P_{n-j}^n
  \sum_{k=3}^j 4L \theta^{j-k} |\bar{e}_k| +
  P_0^n \sum_{k=3}^n 4L \theta^{n-k} |\bar{e}_n| +
  |\bar{e}_2| + 2\sum_{j=3}^n P_{n-j}^n \cdot |\tilde{r}_j(\Delta x)|
 \right).
\eeq
Using \eqref{plow}\eqref{deltax} and $\theta < 2/3$, we have
\[
P_0^n \sum_{k=3}^n 4L \theta^{n-k} < \pi_B \Delta x^{\nu} \cdot 12L \leq \frac{1}{2}.
\]
Thus according to \eqref{en}, we can estimate $\bar{e}_n$ as follows:
\begin{align*}
|\bar{e}_n| &\leq 2\left( \sum_{j=3}^{n-1} P_{n-j}^n
  \sum_{k=3}^j 4L \theta^{j-k} |\bar{e}_k| +
  |\bar{e}_2| + 2\sum_{j=3}^n P_{n-j}^n \cdot |\tilde{r}_j(\Delta x)|
 \right) \\
& \leq 2 \sum_{j=3}^{n-1} P_{n-j}^n
  \sum_{k=3}^j 4L \theta^{j-k} F_k G_k + 2G_n
  \leq 2 \sum_{j=3}^{n-1} P_{n-j}^n
  \sum_{k=3}^j 4L \theta^{j-k} F_j G_n + 2G_n \\
& \leq \left( 24L \sum_{j=3}^{n-1} P_{n-j}^n F_j + 2 \right) G_n
  = \left( 48L \sum_{j=3}^{n-1} P_{n-j}^n E_{\nu}(24\pi_B L x_j^{\nu}) + 2 \right) G_n.
\end{align*}
Finally, we use \eqref{mlineq} to find that
\[
|\bar{e}_n| \leq
  \left( 48\pi_B L \frac{E_{\nu}(24\pi_B L x_n^{\nu})-1}{24\pi_B L} + 2 \right) G_n
  = 2E_{\nu}(24\pi_B L x_n^{\nu}) G_n = F_n G_n.
\]
Thus the lemma is proven by the principle of mathematical induction.
\end{proof}

Our final error estimation can be achieved by combining the above result with
our estimation of the truncation error, and the conclusion is given in the
following theorem:

\begin{theorem}\label{conver}
Let $y$ be the exact solution of \eqref{pb}\eqref{caputo}, and $\{y_k\}_{k=0}^{2N}$ be the numerical solution obtained by \eqref{scheme}.
Assume $y(x)\in C^3[0,T]$. If the step size $\Delta x$ satisfies
\be\label{steps}
\Delta x^{\nu} \leq \frac{1}{24\pi_B L},
\ee
then there exsits a constant $K$ depending on $\nu$, $L$ and the final time $T$, such that
\be\label{ydwx}
|y(x_k)-y_k|\le K\Delta x^{3-\nu}, \quad \text{for all } k=1,2,\cdots,2N.
\ee
\end{theorem}
\begin{proof}
Combining \eqref{p1low}\eqref{rjdx} and \eqref{omega}, we have
\begin{align*}
\sum_{j=3}^k P_{k-j}^k |\tilde{r}_j(\Delta x)| &\leq
  \left( \max_{1\leq j \leq k} \frac{1}{\omega_{1-\nu}(x_j)} \right)
  \sum_{j=1}^k P_{k-j}^k \omega_{1-\nu}(x_j) |\tilde{r}_j(\Delta x)| \\
& \leq \Gamma(1-\nu) x_k^{\nu} \cdot C \Delta x^{3-\nu} \sum_{j=1}^k P_{k-j}^k \omega_{1-\nu}(x_j)
  \leq \left[ C \pi_B \Gamma(1-\nu) x_k^{\nu} \right] \Delta x^{3-\nu}.
\end{align*}
Substituting this estimate into \eqref{ebarn} yields
\beq
\begin{split}
|\bar{e}_n| &\leq 2E_{\nu}(24\pi_B L x_n^{\nu}) \Big(
  |\bar{e}_2| + \left[ 2C \pi_B \Gamma(1-\nu) x_n^{\nu} \right] \Delta x^{3-\nu} \Big) \\
& \leq 2E_{\nu}(24\pi_B L x_n^{\nu}) \Big(
  |e_1| + |e_2| + \left[ 2C \pi_B \Gamma(1-\nu) x_n^{\nu} \right] \Delta x^{3-\nu} \Big) \\
& \leq 2E_{\nu}(24\pi_B L x_n^{\nu})
  \Big( 10C_1 + 2C \pi_B \Gamma(1-\nu) x_n^{\nu} \Big) \Delta x^{3-\nu}, \qquad \text{for all } n \geq 2,
\end{split}
\eeq
where we have used the estimation \eqref{e1pe2}. Therefore, the numerical error $|e_k|$ can be estimated by
\[
|e_k| = \left| \sum_{n=0}^k \theta^{k-n} \bar{e}_n \right|
  \leq 3 \max_{0\leq n \leq k} \bar{e}_n \leq 6E_{\nu}(24\pi_B L x_k^{\nu})
  \Big( 10C_1 + 2C \pi_B \Gamma(1-\nu) x_k^{\nu} \Big) \Delta x^{3-\nu}.
\]
The proof is completed.
\end{proof}
\section {\bf Numerical results}
\label{sec:5}

In this section, we present numerical experiments to verify the theoretical results obtained in the previous sections.

\begin{example}\label{li1}
We consider the problem \eqref{pb} with
\[
f(x,y(x))=\frac{\Gamma(4+\nu)}{6}x^3, \qquad y(0) = 0,
\]
where $f$ is independent of $y$. It can be verified that the exact solution is
$y(x)=x^{3+\nu}$.  The computation is carried out up to $T = 1$. In our tests,
we choose $\nu = 0.3, 0.5, 0.8$ and $0.99$, and for all choices of $\nu$, we
choose the step size to be $\Delta x=\frac{1}{2^l},l=3,4,\cdots,10$. The error
we will display is defined by
\begin{displaymath}
e_{\Delta x} = \max_{k = 1,\cdots,2N} |y(x_k) - y_k|,
\end{displaymath}
where $2N = T / \Delta x$.

By this example, we would like to check the convergence order of the numerical
method with respect to the order of the fractional derivative $\nu$. The
results are given in Table \ref{exa001}, where the convergence order is
computed by $\log_2 (e_{2\Delta x} / e_{\Delta x})$. By Theorem \ref{conver},
we expect that this number is close to $3-\nu$. It is obvious that our
numerical results are consistent with the theoretical analysis.
\begin{table}
\centering
\caption{Maximum error $e_{\Delta x}$ and convergence order for Example \ref{li1}.} \label{exa001}
\begin{tabular}{ccccccccc}
\hline
${\Delta x}$& $\nu=0.3$ & order & $\nu=0.5$ & order & $\nu=0.8$ & order & $\nu=0.99$ & order  \\
\hline
$\frac{1}{8} $ &1.6782e-3& -     &5.8967e-3& -  &2.3580e-2&-&4.7431e-2&-\\
$\frac{1}{16}$ &2.7683e-4&2.5998&1.1467e-3&2.3623&5.8213e-3&2.0181&1.3486e-2&1.8143\\
$\frac{1}{32}$ &4.3876e-5&2.6575&2.1076e-4&2.4438&1.3329e-3&2.1267&3.5413e-3&1.9291\\
$\frac{1}{64}$ &6.8430e-6&2.6807&3.7908e-5&2.4750&2.9674e-4&2.1673&9.0195e-4&1.9731\\
$\frac{1}{128}$&1.0596e-6&2.6910&6.7551e-6&2.4884&6.5272e-5&2.1846&2.2667e-4&1.9924\\
$\frac{1}{256}$&1.6356e-7&2.6957&1.1986e-6&2.4945&1.4278e-5&2.1926&5.6613e-5&2.0014\\
$\frac{1}{512}$&2.5195e-8&2.6986&2.1228e-7&2.4974&3.1153e-6&2.1963&1.4096e-5&2.0057\\
$\frac{1}{1024}$&3.8778e-9&2.6998&3.7565e-8&2.4985&6.7888e-7&2.1981&3.5049e-6&2.0078\\
\hline
\end{tabular}
\end{table}
\end{example}
\begin{example}\label{li2}
In this example, we add the dependence on $y$ to the right-hand side $f(x,y)$.
The following two functions are considered:
\begin{align}
\label{fxx}
f(x,y(x))&= \frac{\Gamma(4+\nu)}{6}x^3+x^{3+\nu}-y(x),\\
\label{ffxx}
f(x,y(x))&= \frac{\Gamma(4+\nu)}{6}x^3+x^{6+2\nu}-y^2(x).
\end{align}
This two right-hand sides correspond to linear and nonlinear dependences on
$y$. With the initial condition $y(0) = 0$, the exact solution of both is
$y(x)=x^{3+\nu}$.

We take $T=1$ again and repeat the calculation in Example \ref{li1}. The
numerical error is provided in Table \ref{exa002} and Table \ref{exa003}.  Due
to the sufficient smoothness of the numerical solution, we again observe a good
agreement with the theoretical convergence order. In particular, it is worth
emphasizing that the non-linearity of $f$ seems to have no impact on the
numerical order of the scheme.

\begin{table}[!ht]
\centering
\caption{Maximum error $e_{\Delta x}$ and convergence order for the right-hand side \eqref{fxx}.} \label{exa002}
\begin{tabular}{ccccccccc}
\hline
${\Delta x}$& $\nu=0.3$ & order & $\nu=0.5$ & order & $\nu=0.8$ & order & $\nu=0.99$ & order \\
\hline
$\frac{1}{8} $&8.9242e-4& -     &3.4577e-3& - &1.6357e-2&-&3.6070e-2&- \\
$\frac{1}{16} $&1.4371e-4&2.6345&6.5136e-4&2.4083&3.9150e-3&2.0628&1.0036e-2&1.8455\\
$\frac{1}{32}$&2.2556e-5&2.6715&1.1826e-4&2.4614&8.8578e-4&2.1439&2.6115e-3&1.9422\\
$\frac{1}{64}$&3.5029e-6&2.6868&2.1163e-5&2.4824&1.9621e-4&2.1744&6.6251e-4&1.9788\\
$\frac{1}{128}$&5.4140e-7&2.6937&3.7628e-6&2.4916&4.3066e-5&2.1878&1.6619e-4&1.9950\\
$\frac{1}{256}$&8.3492e-8&2.6969&6.6703e-7&2.4959&9.4114e-6&2.1940&4.1471e-5&2.0026\\
$\frac{1}{512}$&1.2854e-8&2.6993&1.1806e-7&2.4981&2.0524e-6&2.1970&1.0322e-5&2.0063\\
$\frac{1}{1024}$&1.9781e-9&2.7000&2.0887e-8&2.4989&4.4715e-7&2.1984&2.5659e-6&2.0081\\
\hline
\end{tabular}
\end{table}
\begin{table}[!ht]
\centering
\caption{Maximum error $e_{\Delta x}$ and convergence order for the right-hand side \eqref{ffxx}.} \label{exa003}
\begin{tabular}{ccccccccc}
\hline
${\Delta x}$& $\nu=0.3$ & order & $\nu=0.5$ & order & $\nu=0.8$ & order & $\nu=0.99$ & order \\
\hline
$\frac{1}{8} $&9.1405e-4& -     &3.2126e-3& - &1.5357e-2&-&3.4906e-2&- \\
$\frac{1}{16} $&1.6188e-4&2.4972&6.4829e-4&2.3090&3.8037e-3&2.0134&1.0094e-2&1.7898\\
$\frac{1}{32}$&2.6226e-5&2.6258&1.2091e-4&2.4226&8.7214e-4&2.1247&2.6623e-3&1.9228\\
$\frac{1}{64}$&4.1349e-6&2.6651&2.1873e-5&2.4667&1.9417e-4&2.1672&6.7852e-4&1.9722\\
$\frac{1}{128}$&6.4327e-7&2.6843&3.9072e-6&2.4849&4.2704e-5&2.1848&1.7050e-4&1.9925\\
$\frac{1}{256}$&9.9504e-8&2.6926&6.9413e-7&2.4928&9.3407e-6&2.1927&4.2578e-5&2.0016\\
$\frac{1}{512}$&1.5350e-8&2.6964&1.2299e-7&2.4965&2.0379e-6&2.1964&1.0600e-5&2.0059\\
$\frac{1}{1024}$&2.3643e-9&2.6987&2.1774e-8&2.4978&4.4407e-7&2.1982&2.6356e-6&2.0079\\
\hline
\end{tabular}
\end{table}
\end{example}
\begin{example}\label{li3}
In this example, we consider the eigenvalue problem with right-hand side
$f(x,y)=\lambda y(x)$, where $\lambda < 0$ is a constant. The exact solution is
$y(x)=y_0 E_{\nu}(\lambda x^\nu)$, and $E_{\nu}(\cdot)$ is the Mittag-Leffler
function defined in \eqref{ml}. When $\nu=1$, The fractional derivative reduces
to the ordinary derivative, and the exact solution turns out to be solution is
given as $y(x)=y_0e^{\lambda x}$.

In our test, we set the eigenvalue $\lambda = -1$ and choose the inital value
$y_0 = 1$. The choices of the fractional order are now taken as $\nu = 0.3,
0.6, 0.9$ and $1.0$. Other settings are the same as previous two examples.
When $\nu = 1.0$, our numerical method reduces to the second-order backward
differentiation formula (BDF2). Results are given in Table \ref{exa004}, from
which we can observe that when $\nu < 1$, the convergence order is close to
$\nu$. The reason lies in the singularity of the Mittag-Leffler function at $x
= 0$. When $\nu = 1$, the singularity disappears, and the convergence order $3
- \nu$ is restored.
\begin{table}[!ht]
\centering
\caption{Maximum error $e_{\Delta x}$ and convergence order for Example \ref{li3}.} \label{exa004}
\begin{tabular}{ccccccccc}
\hline
${\Delta x}$& $\nu=0.3$ & order &$\nu=0.6$ & order &$\nu=0.9$ & order &$\nu=1.0$  & order    \\
\hline
$\frac{1}{8}$&3.2510e-3&-&8.8351e-4&-&2.1988e-3&- &3.8804e-4&- \\
$\frac{1}{16} $&2.8864e-3&0.1716&6.6298e-4&0.4143&9.7373e-4&1.1751&2.9709e-4&0.3853\\
$\frac{1}{32}$&2.5263e-3&0.1922&4.6140e-4&0.5229&4.5730e-4&1.0903&9.7657e-5&1.6051\\
$\frac{1}{64}$&2.1840e-3&0.2100&3.1026e-4&0.5725&2.2952e-4&0.9945&2.7213e-5&1.8434\\
$\frac{1}{128}$&1.8684e-3&0.2252&2.0569e-4&0.5930&1.2005e-4&0.9350&7.1461e-6&1.9290\\
$\frac{1}{256}$&1.5842e-3&0.2380&1.3568e-4&0.6003&6.3888e-5&0.9100&1.8289e-6&1.9661\\
$\frac{1}{512}$&1.3332e-3&0.2488&8.9370e-5&0.6023&3.4223e-5&0.9006&4.6252e-7&1.9834\\
$\frac{1}{1024}$&1.1150e-3&0.2578&5.8861e-5&0.6025&1.8362e-5&0.8982&1.1628e-7&1.9918\\
\hline
\end{tabular}
\end{table}

The convergence order can be improved by Lubich's method \cite{lub1986} to
include singular terms in the ansatz of the solution. This is achieved by
choosing a finite sequence of positive real numbers $\sigma_1 < \sigma_2 <
\cdots < \sigma_{m+1}$, and assume that
\beq\label{coy}
y(x)-y(0)=\sum_{j=1}^{m}c_jx^{\sigma_j}+x^{\sigma_{m+1}}\tilde{y}(x),
\eeq
where $\tilde{y}(x)$ is a bounded function, and we assume that the term
$x^{\sigma_{m+1}}\tilde{y}(x)$ is sufficiently smooth to retain our convergence
order. The sum of $c_jx^{\sigma_j}$ captures the less smooth part, for which
the discretization of the fractional derivative needs to be altered to get
better accuracy. Here we omit the detailed derivation, and refer the readers to
\cite{lub1986,zeng2018s,zeng2017s} for more discussions on the correction
method. The final numerical scheme discretizes the fractional derivative by
\beq\label{jzy}
_0D_{\Delta x}^{\nu,m}y_n ={} _0D_{\Delta x}^{\nu}y_n
+{\Delta x}^{-\nu}\sum_{j=1}^{m}W_{n,j}(y_j-y_0),
\eeq
where $W_{n,j}$ are the starting weights that are chosen such that
\beq\label{yyjz}
_0D_{\Delta x}^{\nu} q_k(x_n)
+{\Delta x}^{-\nu}\sum_{j=1}^{m}W_{n,j}q_k(x_j)=\frac{\Gamma(1+\sigma_k)}{\Gamma(1-\nu+\sigma_k)}x_n^{\sigma_k-\nu},
  \qquad \text{for all } k = 1,\cdots,m,
\eeq
where $q_k(x)=x^{\sigma_k}$.

In this example, we choose $\sigma_k=k\nu$. Then $W_{n,j}$, $1\leq j\leq m$ can be solved from \eqref{yyjz}, and the values of $W_{n,j}$ are independent of $\Delta x$. Since the series expansion of the exact solution includes terms such as $x^{\nu}$ and $x^{2\nu}$, Lubich's correction method is suitable for such a problem. The results of the corrected method are given in Table \ref{exa005}, which shows remarkable improvement compared with Table \ref{exa004}.

\begin{table}[!ht]
\centering
\caption{Maximum error and convergence order of the corrected method for Example \ref{li3}.} \label{exa005}
\begin{tabular}{ccccccc}
\hline
${\Delta x}$& $\nu=0.3$ & order &$\nu=0.6$ & order &$\nu=0.9$ & order \\
\hline
$\frac{1}{8}$&2.4932e-6&-&4.2141e-5&-&1.2940e-4&- \\
$\frac{1}{16}$&8.5679e-7&1.5409&1.7729e-5&1.2491&7.0189e-5&8.8254e-01\\
$\frac{1}{32}$&2.8365e-7&1.5947&5.0652e-6&1.8074&2.3691e-5&1.5668\\
$\frac{1}{64}$&9.0097e-8&1.6546&1.2249e-6&2.0479&6.6215e-6&1.8391\\
$\frac{1}{128}$&2.7462e-8&1.7140&2.7037e-7&2.1796&1.6940e-6&1.9667\\
$\frac{1}{256}$&8.0536e-9&1.7697&5.6509e-8&2.2583&4.1466e-7&2.0304\\
$\frac{1}{512}$&2.2805e-9&1.8202&1.1354e-8&2.3152&9.9291e-8&2.0622\\
$\frac{1}{1024}$&6.2613e-10&1.8648&2.5311e-9&2.1654&2.3508e-8&2.0785\\
\hline
\end{tabular}
\end{table}
\end{example}
\section{Conclusion} \label{sec:6}

An efficient high-order approximate numerical scheme for fractional ordinary
differential equations with the Caputo derivative has been introduced in this
paper. The scheme is unconditionally stable and has uniform accuracy for all
time steps. The proof of stability shows the technical details on how to deal
with the special initial steps. The sharp numerical order $3-\nu$ is proven for
sufficiently smooth solutions and general nonlinear equations, and this order
is verified by our numerical experiments.

%

\appendix

\section{Proof of some inequalities} \label{sec:70}
In this appendix, we provide the proofs of two lemmas used in the stability
analysis, which include a number of technical inequalities.
\begin{lemma}\label{xis0}
For any $k\geq2$, it holds that
\begin{align*}
(1)~&  \left( 1 - \frac{1}{k} \right)^{1-\nu} + \left( 1 + \frac{1}{k} \right)^{1-\nu}
  \geq
  2-\frac{1-\nu}{2k^2}\left[2^\nu-\left(\frac{2}{3}\right)^\nu\right],\\
(2)~& \left(1-\frac{1}{k}\right)^{2-\nu}-\left(1+\frac{1}{k}\right)^{2-\nu}\geq
  -2(2-\nu)\frac{1}{k}+\frac{(2-\nu)(1-\nu)\nu}{3k^3},
\\
(3)~& \frac{2-\nu}{2}\frac{1}{2k}\left[\left(1-\frac{2}{2k}\right)^{1-\nu}+3-4\left(1+\frac{1}{2k}\right)^{1-\nu}\right]
+\left(1-\frac{2}{2k}\right)^{2-\nu}-3+2\left(1+\frac{1}{2k}\right)^{2-\nu}\geq0,
\\
(4)~&-\nu^2-12+3\left(\frac{2}{3}\right)^\nu(\nu^2+2\nu+4)>0,
\\
(5)~&6-\nu-\left(2+\frac{\nu}{2}\right)2^\nu3^{1-\nu}<0,
\\
(6)~&{-2}\nu^3+12\nu^2-56\nu-48+3\left(\frac{2}{3}\right)^\nu(3\nu^3+4\nu^2+20\nu+16)<0,
\\
(7)~& 2^{1-\nu} [4-\nu-(2+\nu)2^{1-\nu}] < \frac{1}{27} (2\nu-3)(2-\nu)(1-\nu)\nu,
\\
(8)~&12-\nu^2-(12+8\nu+\nu^2)2^{-\nu} > \frac{1}{16} (2+\nu)(2-\nu)(1-\nu)\nu.
\end{align*}
\end{lemma}
\begin{proof}
(1) This inequality is equivalent to
\beq \label{f1t}
f_1(t) := (1-t)^{1-\nu}+(1+t)^{1-\nu} - 2(1-A t^2) \geq 0
\eeq
for $t = 1/k$ and $A = \frac{1}{4}(1-\nu) [2^\nu - (2/3)^\nu]$. Since $k \geq 2$, the range of $t$ is $(0,1/2]$. To show \eqref{f1t}, we take the derivative of $f_1(t)$ to get
\beq \label{df1t}
f_1'(t) = -(1-\nu)[(1-t)^{-\nu}-(1+t)^{-\nu}]+4At = -4t[f_2(t) - A],
\eeq
where
\[
f_2(t) = \frac{1-\nu}{4} \frac{(1-t)^{-\nu}-(1+t)^{-\nu}}{t}
  = \frac{1-\nu}{4} \sum_{j=0}^{+\infty} \frac{2}{(2j+1)!} \nu(1+\nu)\cdots(2j+\nu) t^{2j}.
\]
The series expansion of $f_2$ clearly shows that $f_2$ is an increasing function, which yields $f_2(t) \leq f_2(1/2) = A$. Thus by \eqref{df1t}, we have $f_1'(t) \geq 0$, indicating that
\[
f_1(t) \geq f_1(0) = 0.
\]

(2) This inequality can be similarly proven by defining
\[
f_3(t)=(1-t)^{2-\nu}-(1+t)^{2-\nu}+2(2-\nu)t-\frac{(2-\nu)(1-\nu)\nu}{3} t^3,
\]
whose series expansion is
\[
f_3(t) = \sum_{j=2}^{+\infty} \frac{2}{(2j+1)!} (-2+\nu)(-1+\nu)\nu(1+\nu) \cdots (2j-2+\nu) t^{2j+1}.
\]
Since all the terms in the sum are monotonically increasing, we have $f_3(t) \geq f_3(0) = 0$. The proof is completed by setting $t = 1/k$.

(3) Let
\beq\label{f12k}
f_1(k,\nu)=\left(1-\frac{2}{2k}\right)^{2-\nu}+2\left(1+\frac{1}{2k}\right)^{2-\nu}, \qquad
f_2(k,\nu)=\left(1-\frac{2}{2k}\right)^{2-\nu}.
\eeq
Then the desired inequality is equivalent to
\[
f(k,\nu):=\left(1-\frac{2-\nu}{2k+1}\right)f_1(k,\nu)+
  \left[\frac{2-\nu}{4(k-1)}+\frac{2-\nu}{2k+1}\right]f_2(k,\nu)+\frac{3(2-\nu)}{4k}-3\geq0.
\]
Since $k \geq 2$, we can apply binomial expansion to $f_1$ to obtain
\[
f_1(k,\nu) = \sum_{j=0}^{+\infty}
  {2-\nu \choose j} [(-2)^j + 2] \left( \frac{1}{2k} \right)^j.
\]
It can be observed that when $j \geq 2$, the summand in the above sum is positive. Therefore
\beq \label{f1kn}
f_1(k,\nu) \geq \sum_{j=0}^3
  {2-\nu \choose j} [(-2)^j + 2] \left( \frac{1}{2k} \right)^j
= 3 + \frac{3(2-\nu)(1-\nu)}{4k^2} + \frac{(2-\nu)(1-\nu)\nu}{8k^3}.
\eeq
By similar means, we get
\beq \label{f2kn}
f_2(k,\nu) \geq 1 - \frac{2-\nu}{k} + \frac{(2-\nu)(1-\nu)}{2k^2} + \frac{(2-\nu)(1-\nu)\nu}{6k^3}.
\eeq
Plugging \eqref{f1kn} and \eqref{f2kn} into the expression of the $f(k,\nu)$, we get
\begin{align*}
f(k,\nu)&\geq \frac{(2-\nu)(1-\nu)\nu}{8k^3(k-1)(2k+1)} [-3k+(k-1)(2k-1+\nu)+(2k-1)(2-\nu)] \\
& \geq \frac{(2-\nu)(1-\nu)\nu}{8k^3(k-1)(2k+1)} [-3k+(k-1)(3+\nu)+3(2-\nu)] \\
& = \frac{(2-\nu)(1-\nu)\nu}{8k^3(k-1)(2k+1)} [(k-1)\nu+3(1-\nu)] \geq 0,
\end{align*}
which completes the proof.

(4) Let $\tilde{h}(\nu)=-\nu^2-12+3(\frac{2}{3})^\nu(\nu^2+2\nu+4)$. Then its first-order derivative is
\[
\tilde{h}'(\nu) = -2\nu+3\left(\frac{2}{3}\right)^\nu \left[2\nu+2+(\nu^2+2\nu+4)\log\frac{2}{3} \right].
\]
When $\nu \in (0,1)$, by the property of quadratic functions, one can show that $2\nu+2+(\nu^2+2\nu+4)\log(2/3) > 0$. Therefore
\[
\tilde{h}'(\nu) \geq -2\nu+2 \left[2\nu+2+(\nu^2+2\nu+4)\log\frac{2}{3} \right]
  = 2 \left[\nu+2+(\nu^2+2\nu+4)\log\frac{2}{3} \right].
\]
Again it can be shown by the property of quadratic functions that $\tilde{h}'(\nu) > 0$ for all $\nu \in (0,1)$. Thus
\[
\tilde{h}(\nu) > \tilde{h}(0) = 0.
\]

(5) Let $g(\nu)=6-\nu-(2+\nu/2)2^\nu3^{1-\nu}$. We want to prove $g(\nu)<0$. The first-order and second-order derivatives of $g$ are
\begin{align*}
g'(\nu) &= -1-\left(\frac{2}{3}\right)^{\nu-1}-\left(6+\frac{3}{2}\nu\right)\left(\frac{2}{3}\right)^\nu\log\frac{2}{3}, \\
g''(\nu) &= -\left( \frac{2}{3} \right)^{\nu-1}
  \left( 2+\log\frac{16}{81}+\nu \log \frac{2}{3} \right) \log \frac{2}{3}.
\end{align*}
It is clear that $g''(\nu)$ changes from positive to negative as $\nu$ varies from $0$ to $1$. Therefore $g'(\nu)$ first increases and then decreases. By straightforward calculation, we see that $g'(0) < 0$ and $g'(1) > 0$, meaning that $g(\nu)$ first decreases and then increases. Therefore
\begin{displaymath}
g(\nu) < \max(g(0), g(1)) = 0.
\end{displaymath}

(6) Let $\tilde{g}(\nu)$ be the left-hand side. Then the fourth-order derivative of $\tilde{g}(\nu)$ is
\begin{align*}
\tilde{g}^{(4)}(\nu) &= 3\left( \frac{2}{3} \right)^{\nu} \log \frac{2}{3} \Bigg[
  72 + 6(18 \nu + 8) \log \frac{2}{3} \\
& \quad + 4(9\nu^2 + 8\nu + 20) \left( \log \frac{2}{3} \right)^2 +
  (3\nu^3 + 4\nu^2 + 20\nu + 16)\left( \log \frac{2}{3} \right)^3
\Bigg] \\
& < 3\left( \frac{2}{3} \right)^{\nu} \log \frac{2}{3}
  \left[ 72 + 156 \log \frac{2}{3} + 80\left( \log \frac{2}{3} \right)^2
    + 43 \left( \log \frac{2}{3} \right)^3 \right] < 0.
\end{align*}
Therefore $\tilde{g}'''(\nu)$ is monotonically decreasing for $\nu \in (0,1)$. Straightforward calculation yields $\tilde{g}'''(0) > 0$ and $\tilde{g}'''(1) < 0$, which indicates that $\tilde{g}''(\nu)$ first increases and then decreases. Since $\tilde{g}''(0) > 0$ and $\tilde{g}''(1) > 0$, we know that $\tilde{g}'(\nu)$ increases monotonically. Finally, using $\tilde{g}'(0) < 0$ and $\tilde{g}'(1) > 0$, one sees that $\tilde{g}(\nu)$ first decreases and then increases, which implies
\begin{displaymath}
\tilde{g}(\nu) < \max(\tilde{g}(0), \tilde{g}(1)) = 0.
\end{displaymath}
This completes the proof.

(7) This inequality can be proven using the same method as (6).

(8) Define
\[
g(\nu) = \frac{1}{16} (2+\nu)(2-\nu)(1-\nu)\nu - [12-\nu^2-(12+8\nu+\nu^2)2^{-\nu}].
\]
The third-order derivative of $g$ satisfies
\begin{align*}
g^{(3)}(\nu) &= 2^{-3-\nu} \left[
  2^{\nu} (12\nu - 3) - 48\log 2 + (192 + 48\nu) (\log 2)^2 - (8\nu^2 + 64\nu + 96) (\log 2)^3
\right] \\
& > 2^{-3-\nu} [-3 - 48 \log 2 + 192 (\log 2)^2 - 168 (\log 2)^3] > 0,
\end{align*}
which means $g''(\nu)$ is monotonically increasing. Using $g''(1) < 0$, we know that $g'(\nu)$ is an decreasing function. Finally, using $g'(0) < 0$, we know that $g'(\nu)$ is negative for all $\nu \in (0,1)$. Thus $g(\nu) < g(0) = 0$.
\end{proof}

\begin{lemma} \label{gineq}
Suppose $0 < b < 2m$. Let
\[
f(\nu) = (2-\nu) [a_1 (2m)^{1-\nu} + a_2 (2m+b)^{1-\nu}] +
  a_3 [(2m)^{2-\nu} - (2m+b)^{2-\nu}].
\]
Then we have
\begin{enumerate}
\item If $a_3 b / a_2 \leq 3/2$ and $a_2 < 0$, then
  \beq \label{lem}
  \begin{split}
  f(\nu) &< (2-\nu) (2m)^{1-\nu} \left[
    a_1 + a_2 - a_3 b + \sum_{k=1}^2 {1-\nu \choose k} \left( a_2 - \frac{a_3 b}{k+1} \right) \frac{b^k}{(2m)^k} \right] \\
  &< (2-\nu) (2m)^{1-\nu} (a_1 + a_2 - a_3 b).
  \end{split}
  \eeq
\item If $a_3 b = 2a_2 > 0$, then
  \beq \label{second_case}
  f(\nu) < (2-\nu) (2m)^{1-\nu} \left[
  a_1 + a_2 - a_3 b - a_2\left( \frac{b}{2m} \right)^2 \frac{(1-\nu)\nu}{6}
    \left( 1 - \frac{\nu+1}{2} \frac{b}{2m} \right)
  \right].
  \eeq
\end{enumerate}
\end{lemma}

\begin{proof}
Since $b < 2m$, we can apply binomial expansion to get
\beq \label{binexp}
\begin{split}
f(\nu) &= (2-\nu) (2m)^{1-\nu} \left[ a_1 + a_2 \left( 1 + \frac{b}{2m} \right)^{1-\nu} \right]
  + a_3 (2m)^{2-\nu} \left[ 1 - \left( 1 + \frac{b}{2m} \right)^{2-\nu} \right] \\
&= (2-\nu) (2m)^{1-\nu} \left[ a_1 + a_2 - a_3 b + \sum_{k=1}^{+\infty}
  {1-\nu \choose k} \left( a_2 - \frac{a_3 b}{k+1} \right) \frac{b^k}{(2m)^k} \right].
\end{split}
\eeq
When $0 < b < 2m$ and $a_3 b / a_2 \leq 2$, then the above series is an alternating series. Denote the above series by $\sum_{k=1}^{+\infty} S_k$. Then by $b < 2m$, we see that
\[
|S_{k+1}| \leq \frac{(k+\nu-1)(k+2-a_3 b / a_2)}{(k+2)(k+1-a_3 b / a_2)} |S_k|.
\]
We want to show that the factor in front of $|S_k|$ is less than one, meaning that $\{|S_k|\}$ decreases monotonically. To show this, we take the difference between the numerator and the denominator:
\beq \label{num_den}
(k+\nu-1)(k+2-a_3 b / a_2) - (k+2)(k+1-a_3 b / a_2) = (2+k)(\nu-2) + \frac{a_3 b}{a_2} (3-\nu).
\eeq
Now we consider the two cases separately:

Case 1: If $a_3 b / a_2 \leq 3/2$ and $a_2 < 0$, then
\[
(2+k)(\nu-2) + \frac{a_3 b}{a_2} (3-\nu) \leq 3(\nu-2) + \frac{3}{2}(3-\nu) = \frac{3}{2}(\nu-1) < 0.
\]
Therefore $|S_{k+1}| \leq |S_k|$, indicating that the sign of the alternating series is determined by the sign of the first term. Using
\[
S_1 = (1-\nu) \left( a_2 - \frac{a_3 b}{2} \right) \frac{b}{2m}
  = (1-\nu) a_2 \left( 1 - \frac{a_3 b}{2a_2} \right) \frac{b}{2m} < 0,
\]
we conclude that the series in \eqref{binexp} is negative. Therefore
\[
f(\nu) < (2-\nu) (2m)^{1-\nu} (a_1 + a_2 - a_3 b + S_1 + S_2).
\]

Case 2: If $a_3 b = 2a_2 > 0$, we have $S_1 = 0$. We only need to study the sign of \eqref{num_den} when $k \geq 2$:
\[
(2+k)(\nu-2) + \frac{a_3 b}{a_2} (3-\nu) \leq 4(\nu-2) + 2(3-\nu) = 2(\nu-1) < 0.
\]
Therefore we also have $|S_{k+1}| \leq |S_k|$. Now the first term in the series is $S_2 = -(\frac{b}{2m})^2 \frac{(1-\nu)\nu}{3!} a_2 < 0$. Thus the whole series is also negative. In this case, we have
\[
f(\nu) < (2-\nu) (2m)^{1-\nu} (a_1 + a_2 - a_3 b + S_2 + S_3).
\]
The equation \eqref{second_case} can be obtained by inserting the expressions of $S_2$ and $S_3$.
\end{proof}

\section{Some results in the proof of Lemma \ref{ebar}} \label{sec:80}
Now we provide the proof of some results used in the proof of Lemma \ref{ebar}. The
proof is generally in accordance with the corresponding results in
\cite{liao2019dis}. The difference is that according to our definition of
$P_j^n$, the equations \eqref{mdengyu1} and \eqref{m2dengyu1} are not
equalities. Consequently, the results in \cite{liao2019dis} cannot be directly
applied to our case. Below we divide the proof into three lemmas.
\begin{lemma}\label{Pnjk}
The discrete kernels $P_{j}^{n}$ defined in \eqref{Pjn} satisfy
\begin{gather} \label{plow}
0\leq P_{n-j}^{n}\leq \pi_B\Gamma(2-\nu)\Delta x^{\nu},
  \qquad 3\leq j\leq n\leq 2N,\\\label{p1low}
\sum_{j=3}^{n}P_{n-j}^{n}\omega_{1-\nu}(x_j)\leq\pi_B,
  \qquad 3\leq n\leq 2N.
\end{gather}

\end{lemma}
\begin{proof}
According \cite[Lemma 2.1]{liao2019dis}, we can directly obtain (\ref{plow}).
We mainly focus on \eqref{p1low}. Taking $n=j$ and $k=j-1$ in \eqref{dyxj1},
we have
\beq
\begin{split}
\bar{B}_{j-1}^{j} & \geq\frac{1}{\pi_B\Delta x \Gamma(2-\nu)} [j^{1-\nu} - (j-1)^{1-\nu}]
= \frac{1}{\pi_B\Delta x \Gamma(1-\nu)} \frac{j^{1-\nu} - (j-1)^{1-\nu}}{1-\nu} \\
&\geq\frac{1}{\pi_B\Delta x^{\nu}\Gamma(1-\nu)}j^{-\nu}
=\frac{1}{\pi_B}\frac{x_j^{-\nu}}{\Gamma(1-\nu)}=\frac{1}{\pi_B}\omega_{1-\nu}(x_j),
\end{split}
\eeq
which indicates $\omega_{1-\nu}(x_j)\leq\pi_B\bar{B}_{j-1}^{j}$. By \eqref{pnj}
and Lemma \ref{pib}, we obtain
\[
\sum_{j=3}^{n}P_{n-j}^{n}\omega_{1-\nu}(x_j)
\leq\pi_B\sum_{j=3}^{n}P_{n-j}^{n}\bar{B}_{j-1}^{j}
\leq\pi_B\sum_{j=3}^{n}P_{n-j}^{n}\bar{B}_{j-3}^{j}
=\pi_B,
\]
as completes the proof.
\end{proof}
\begin{lemma}\label{vdex1z}
Let $v:[0,T]\rightarrow R$ be a continuous and piecewise $C^1$ function whose derivative $v'(x)$ is nonnegative for all $x \in [0,T]$. Then

(I) If $v'$ is monotonically decreasing, we have
\beq\label{pnjjn}
\sum_{j=3}^{n}P_{n-j}^{n}(_0D^{\nu}_x v)(x_j)\leq\pi_B\int_{0}^{x_n}v'(s)ds=\pi_B[v(x_n)-v(0)],
 \qquad 3\leq n\leq 2N.
\eeq

(II) If $v'$ is monotonic, then
\beq\label{p1njjn}
\sum_{j=3}^{n-1}P_{n-j}^{n}(_0D^{\nu}_x v(x_j))\leq\pi_B\int_{0}^{x_n}v'(s)ds=\pi_B[v(x_n)-v(0)],
  \qquad 3\leq n\leq 2N.
\eeq
\end{lemma}
\begin{proof}
(I) The proof requires the Chebyshev's sorting inequality \cite[P.168, item 236]{GHardy1934}:
if $f$ is monotone increasing and $g$ is monotone decreasing on the interval $[a,b]$, and both functions are integrable, we have
\[
(b-a)\int^b_a f(s)g(s)ds\leq \int^b_a f(s)ds \int^b_a g(s)ds.
\]

In this inequality, we set $[a,b]=[x_{k-1},x_k]$, $f(s)=w_{1-\alpha}(x_j-s)$ and $g(s)=v'(s)\geq 0$. Using Lemma \ref{pib}, we see that when $j\geq 3$,
\beq\label{FSJ}
\begin{split}
(_0D^{\nu}_x v)(x_j) &= \int^{x_j}_0 w_{1-\alpha}(x_j-s) v'(s)ds=\sum^j_{k=1} \int^{x_k}_{x_{k-1}}w_{1-\alpha}(x_j-s) v'(s)ds\\
&\leq \sum^j_{k=1} \frac{1}{\Delta x}\int^{x_k}_{x_{k-1}}w_{1-\alpha}(x_j-s) ds\int^{x_k}_{x_{k-1}}v'(s) ds\\
&\leq  \sum^j_{k=1} \pi_B \bar{B}^j_{j-k}\int^{x_k}_{x_{k-1}}v'(s) ds=\pi_B \sum^j_{k=1}  \bar{B}^j_{j-k}\int^{x_k}_{x_{k-1}}v'(s) ds.
\end{split}
\eeq
Thus, from the \eqref{pnj}, \eqref{mdengyu1} and \eqref{m2dengyu1}, we conclude that
\begin{align*}
\sum^n_{j=3}P^n_{n-j}(_0D^{\nu}_x v)(x_j) &\leq \sum^n_{j=3}P^n_{n-j}\pi_B \sum^j_{k=1}  \bar{B}^j_{j-k}\int^{x_k}_{x_{k-1}}v'(s) ds\\
&= \pi_B\sum^n_{k=1}\int^{x_k}_{x_{k-1}}v'(s) ds \sum^n_{j=k}P^n_{n-j}\bar{B}^j_{j-k}
 \leq \pi_B\int^{x_n}_0v'(s) ds.
\end{align*}

(II) Since $v'(x) \geq 0$, we have
\[
(_0D^{\nu}_x v)(x_j) =\sum^j_{k=1} \int^{x_k}_{x_{k-1}}w_{1-\alpha}(x_j-s) v'(s)ds\geq0.
\]
Therefore if $v'$ is monotonically decreasing, then \eqref{p1njjn} is a simple corollary of \eqref{pnjjn}.
If $v'$ is increasing, we can use Lemma \ref{pib} and \eqref{pnj2j1} to obtain
\begin{align*}
&\sum_{j=3}^{n-1}P_{n-j}^{n}(_0D^{\nu}_x v)(x_j)
=\sum_{j=3}^{n-1}P_{n-j}^{n}\sum^j_{k=1} \int^{x_k}_{x_{k-1}}w_{1-\alpha}(x_j-s) v'(s)ds\\
\leq{} &\sum_{j=3}^{n-1}P_{n-j}^{n}\sum^j_{k=1}v'(x_k) \int^{x_k}_{x_{k-1}}w_{1-\alpha}(x_j-s)ds
\leq\pi_B\Delta x\sum_{j=3}^{n-1}P_{n-j}^{n}\sum^j_{k=1}v'(x_k)\bar{B}^j_{j-k}\\
={} &\pi_B\Delta x\sum_{j=1}^{n-1}P_{n-j}^{n}\sum^j_{k=1}v'(x_k)\bar{B}^j_{j-k}
=\pi_B\Delta x\sum^{n-1}_{k=1}v'(x_k)\sum_{j=k}^{n-1}P_{n-j}^{n}\bar{B}^j_{j-k}\\
\leq{} &\pi_B\Delta x\sum^{n-1}_{k=1}v'(x_k)
\leq\pi_B\sum^{n-1}_{k=1}\int_{x_k}^{x_{k+1}}v'(s)ds
\leq\pi_B\int_{0}^{x_{n}}v'(s)ds.
\end{align*}
This proves \eqref{p1njjn}.
\end{proof}
\begin{lemma}\label{vdexz}
For the discrete kernels $P_{j}^{n}$ defined in \eqref{Pjn}, it holds for any $\mu > 0$ that
\beq \label{mlineq}
\sum_{j=3}^{n-1} P_{n-j}^n E_{\nu}(\mu x_j^{\nu}) \leq \frac{\pi_B}{\mu} [E_{\nu} (\mu x_n^{\nu}) - 1],
  \qquad 3\leq n\leq 2N,
\eeq
where $E_{\nu}(\cdot)$ is the Mittag-Leffler function defined by
\beq \label{ml}
E_{\nu}(z) := \sum_{k=0}^{+\infty} \frac{z^k}{\Gamma(1 + k\nu)}.
\eeq
\end{lemma}
\begin{proof}
Define $v_k(x) = x^{k\nu} / \Gamma(1+k\nu)$. Then
\beq\label{enu}
E_{\nu}(\mu x^{\nu})=\sum_{k=0}^{+\infty} \frac{\mu^k x^{k\nu}}{\Gamma(1 + k\nu)}
= \sum_{k=0}^{+\infty}\mu^kv_k(x).
\eeq
The function $v_k(x)$ satisfies
\begin{gather} \label{vkt}
v_0(x)=1, \qquad 
v'_k(x)=\frac{x^{k\nu-1}}{\Gamma(k\nu)}=\omega_{k\nu}(x), \qquad
v''_k(x)=\frac{(k\nu-1)x^{k\nu-2}}{\Gamma(k\nu)}, \\
\label{vkm1}
(_0D^{\nu}_x v_k)(x)
=\int^{x}_0 w_{1-\alpha}(x_j-s)\omega_{k\nu}(s)ds=\omega_{1+(k-1)\nu}(x)
=v_{k-1}(x),\quad \forall k\geq1.
\end{gather}
Therefore for all $x>0$, $v''_k(x)\leq0$ if $k\nu-1\leq0$ and $v''_k(x)>0$ if $k\nu-1>0$.
Thus, $v'_k(x)$ is non-negative and monotonic, so we can apply \eqref{p1njjn} to get
\beq\label{ppj3n}
\sum_{j=3}^{n-1}P_{n-j}^{n}(_0D^{\nu}_x v_k)(x_j)\leq \pi_B[v_k(x_n)-v_k(0)]
=\pi_Bv_k(x_n), \qquad \forall k\geq1.
\eeq
The equations \eqref{vkm1} and \eqref{ppj3n} yield
\beq\label{kdengyu3}
\sum_{j=3}^{n-1}P_{n-j}^{n}\sum_{k=1}^{m}\mu^kv_{k-1}(x_j)\leq\pi_B\sum_{k=1}^{m}\mu^kv_k(x_n).
\eeq
Now we take the limit $m\rightarrow+\infty$. The right-hand side of the above
inequality approaches to $\pi_B (E_{\nu} (\mu x_n^{\nu}) - 1)$, and the limit
of the left-hand side is
\[
\sum_{j=3}^{n-1}P_{n-j}^{n}\sum_{k=1}^{+\infty}\mu^kv_{k-1}(x_j)=
  \sum_{j=3}^{n-1}P_{n-j}^{n}\mu\sum_{k=0}^{+\infty}\mu^kv_k(x_j)=
  \sum_{j=3}^{n-1}P_{n-j}^{n}\mu E_{\nu}(\mu x_j^{\nu}).
\]
This completes the proof.
\end{proof}

\bibliographystyle{siamplain}
\bibliography{ref2}

\begin{thebibliography}{10}

\bibitem{alikh2015}
{\sc A.~A. Alikhanov}, {\em A new difference scheme for the time fractional
  diffusion equation}, Journal of Computational Physics, 280 (2015),
  pp.~424--438.

\bibitem{bafft2017}
{\sc D.~Baffet and J.~S. Hesthaven}, {\em High-order accurate local schemes for
  fractional differential equations}, Journal of Scientific Computing, 70
  (2017), pp.~355--385.

\bibitem{cao2015high}
{\sc J.~Cao, C.~Li, and Y.~Chen}, {\em High-order approximation to caputo
  derivatives and caputo-type advection-diffusion equations (ii)}, Fractional
  Calculus and Applied Analysis, 18 (2015), pp.~735--761.

\bibitem{LWD2014}
{\sc {Changpin Li, Rifang Wu, Hengfei Ding}}, {\em {High-order approximation to
  Caputo derivatives and Caputo-type advection-diffusion equations}},
  Communications in Applied and Industrial Mathematics, e-536 (2014),
  pp.~1--32.

\bibitem{diethelm2006}
{\sc K.~Diethelm, J.~M. Ford, N.~J. Ford, and M.~Weilbeer}, {\em Pitfalls in
  fast numerical solvers for fractional differential equations}, Journal of
  computational and applied mathematics, 186 (2006), pp.~482--503.

\bibitem{dieelm2002}
{\sc K.~Diethelm and N.~J. Ford}, {\em Analysis of fractional differential
  equations}, Journal of Mathematical Analysis and Applications, 265 (2002),
  pp.~229--248.

\bibitem{du2019}
{\sc R.~Du, Y.~Yan, and Z.~Liang}, {\em A high-order scheme to approximate the
  caputo fractional derivative and its application to solve the fractional
  diffusion wave equation}, Journal of Computational Physics, 376 (2019),
  pp.~1312--1330.

\bibitem{Gao2017SJ}
{\sc {G. Gao, Z. Sun}}, {\em A compact finite difference scheme for the
  fractional sub-diffusion equations}, J. Comput. Phys., 230 (2011),
  pp.~586--595.

\bibitem{GSZ2014}
{\sc {G. Gao, Z. Sun,H. Zhang}}, {\em {A new fractional numerical
  differentiation formula to approximate the Caputo fractional derivative and
  its applications}}, J. Comput. Phys., 259 (2014), pp.~33--50.

\bibitem{GHardy1934}
{\sc {G. H. Hardy, J. E. Littlewood, and G. P{\'o}lya}}, {\em Inequalities},
  Cambridge University Press, 1934.

\bibitem{garrappa2018}
{\sc R.~Garrappa, E.~Messina, and A.~Vecchio}, {\em Effect of perturbation in
  the numerical solution of fractional differential equations.}, Discrete \&
  Continuous Dynamical Systems-Series B, 23 (2018).

\bibitem{guo2015}
{\sc B.~Guo, X.~Pu, and F.~Huang}, {\em Fractional partial differential
  equations and their numerical solutions}, World Scientific, 2015.

\bibitem{HTV11}
{\sc J.~Huang, Y.~Tang, and L.~V\'azquez}, {\em Convergence analysis of a
  block-by-block method for fractional differential equations}, Numer. Math.
  Theor. Meth. Appl., 5 (2012), pp.~229--241.

\bibitem{JCX2013}
{\sc {J. Cao, C. Xu}}, {\em {A high order schema for the numerical solution of
  the fractional ordinary differential equations}}, J. Comput. Phys., 238
  (2013), pp.~154--168.

\bibitem{jiang2017}
{\sc S.~Jiang, J.~Zhang, Q.~Zhang, and Z.~Zhang}, {\em Fast evaluation of the
  caputo fractional derivative and its applications to fractional diffusion
  equations}, Communications in Computational Physics, 21 (2017), pp.~650--678.

\bibitem{JLZ2018}
{\sc B.~Jin, B.~Li, and Z.~Zhou}, {\em Numerical analysis of nonlinear
  subdiffusion equations}, Numerical analysis of nonlinear subdiffusion
  equations, 56 (2018), pp.~1--23.

\bibitem{kilb2006t}
{\sc A.~A.~A. Kilbas, H.~M. Srivastava, and J.~J. Trujillo}, {\em Theory and
  applications of fractional differential equations}, vol.~204, Elsevier
  Science Limited, 2006.

\bibitem{KA06}
{\sc P.~Kumar and O.~Agrawal}, {\em {An approximate method for numerical
  solution of fractional differential equations}}, Signal Process., 86 (2006),
  pp.~2602--2610.

\bibitem{li2015n}
{\sc C.~Li and F.~Zeng}, {\em Numerical methods for fractional calculus},
  Chapman and Hall/CRC, 2015.

\bibitem{Li2018}
{\sc D.~Li, H.-L. Liao, W.~Sun, J.~Wang, and J.~Zhang}, {\em Analysis of
  {$L$}1-{G}alerkin {FEM}s for time-fractional nonlinear parabolic problems},
  Commun. Comput. Phys., 24 (2018), pp.~86--103.

\bibitem{li2016high}
{\sc H.~Li, J.~Cao, and C.~Li}, {\em High-order approximation to caputo
  derivatives and caputo-type advection--diffusion equations (iii)}, Journal of
  computational and Applied mathematics, 299 (2016), pp.~159--175.

\bibitem{LLZ2018}
{\sc H.-L. Liao, D.~Li, and J.~Zhang}, {\em Sharp error estimate of the
  nonuniform {L}1 formula for linear reaction-subdiffusion equations}, SIAM J.
  Numer. Anal., 56 (2018), pp.~1112--1133.

\bibitem{liao2019dis}
{\sc H.-l. Liao, W.~McLean, and J.~Zhang}, {\em A discrete gronwall inequality
  with applications to numerical schemes for subdiffusion problems}, SIAM
  Journal on Numerical Analysis, 57 (2019), pp.~218--237.

\bibitem{Liao2019}
{\sc H.-L. Liao, Y.~Yan, and J.~Zhang}, {\em Unconditional convergence of a
  fast two-level linearized algorithm for semilinear subdiffusion equations},
  J. Sci. Comput., 80 (2019), pp.~1--25.

\bibitem{LIP85}
{\sc P.~Linz}, {\em {Analytical and numerical methods for Volterra equations}},
  Society for Industrial Mathematics, Philadelphia, 1985.

\bibitem{lub1986}
{\sc C.~Lubich}, {\em Discretized fractional calculus}, SIAM J. Math. Anal., 17
  (1986), pp.~704--719.

\bibitem{luk2013prop}
{\sc Y.~Luchko, F.~Mainardi, and Y.~Povstenko}, {\em Propagation speed of the
  maximum of the fundamental solution to the fractional diffusion--wave
  equation}, Computers \& Mathematics with Applications, 66 (2013),
  pp.~774--784.

\bibitem{Luo2017}
{\sc W.-H. Luo, C.~Li, T.-Z. Huang, X.-M. Gu, and G.-C. Wu}, {\em A high-order
  accurate numerical scheme for the {C}aputo derivative with applications to
  fractional diffusion problems}, Numer. Func. Anal. Opt., 39 (2017),
  pp.~600--622.

\bibitem{lv2016error}
{\sc C.~Lv and C.~Xu}, {\em Error analysis of a high order method for
  time-fractional diffusion equations}, SIAM Journal on Scientific Computing,
  38 (2016), pp.~A2699--A2724.

\bibitem{maindi2010}
{\sc F.~Mainardi}, {\em Fractional calculus and waves in linear
  viscoelasticity: an introduction to mathematical models}, World Scientific,
  2010.

\bibitem{miller1993}
{\sc K.~S. Miller and B.~Ross}, {\em An introduction to the fractional calculus
  and fractional differential equations}, Wiley-Interscience, 1993.

\bibitem{nguyn2017}
{\sc T.~B. Nguyen and B.~Jang}, {\em A high-order predictor-corrector method
  for solving nonlinear differential equations of fractional order}, Fractional
  Calculus and Applied Analysis, 20 (2017), pp.~447--476.

\bibitem{vog2017}
{\sc S.~Vong, C.~Shi, and P.~Lyu}, {\em High-order compact schemes for
  fractional differential equations with mixed derivatives}, Numerical Methods
  for Partial Differential Equations, 33 (2017), pp.~2141--2158.

\bibitem{lin2007}
{\sc {Y. Lin, C. Xu}}, {\em Finite difference/spectral approximations for the
  time-fractional diffusion equation}, J. Comput. Phys., 225 (2007),
  pp.~1533--1552.

\bibitem{YOA54}
{\sc A.~Young}, {\em The application of approximate product-integration to the
  numerical solution of integral equations}, Proc. Roy. Soc. London Ser. A, 224
  (1954), pp.~561--573.

\bibitem{zeng2018s}
{\sc F.~Zeng, I.~Turner, and K.~Burrage}, {\em A stable fast time-stepping
  method for fractional integral and derivative operators}, Journal of
  Scientific Computing, 77 (2018), pp.~283--307.

\bibitem{zeng2017s}
{\sc F.~Zeng, Z.~Zhang, and G.~E. Karniadakis}, {\em Second-order numerical
  methods for multi-term fractional differential equations: smooth and
  non-smooth solutions}, Computer Methods in Applied Mechanics and Engineering,
  327 (2017), pp.~478--502.

\end{thebibliography}
\end{document}